\documentclass[11pt]{article}

\usepackage[margin=1in]{geometry}
\usepackage{amsmath,amssymb,amsfonts,mathrsfs}
\usepackage{xcolor}
\usepackage[normalem]{ulem}
\usepackage[colorlinks=true,linkcolor=blue,citecolor=blue,urlcolor=blue]{hyperref}

\numberwithin{equation}{section}

\newcounter{theorem}[section]
\renewcommand{\thetheorem}{\thesection.\arabic{theorem}}
\newenvironment{theorem}[1][]{%
  \refstepcounter{theorem}\par\medskip\noindent
  \textbf{Theorem~\thetheorem\if\relax\detokenize{#1}\relax\else\ (#1)\fi.}\itshape\ }%
  {\par\medskip}
\newenvironment{proposition}[1][]{%
  \refstepcounter{theorem}\par\medskip\noindent
  \textbf{Proposition~\thetheorem\if\relax\detokenize{#1}\relax\else\ (#1)\fi.}\itshape\ }%
  {\par\medskip}
\newenvironment{lemma}[1][]{%
  \refstepcounter{theorem}\par\medskip\noindent
  \textbf{Lemma~\thetheorem\if\relax\detokenize{#1}\relax\else\ (#1)\fi.}\itshape\ }%
  {\par\medskip}
\newenvironment{corollary}[1][]{%
  \refstepcounter{theorem}\par\medskip\noindent
  \textbf{Corollary~\thetheorem\if\relax\detokenize{#1}\relax\else\ (#1)\fi.}\itshape\ }%
  {\par\medskip}
\newenvironment{remark}[1][]{%
  \refstepcounter{theorem}\par\medskip\noindent
  \textbf{Remark~\thetheorem\if\relax\detokenize{#1}\relax\else\ (#1)\fi.}\ }%
  {\par\medskip}
\newenvironment{proof}[1][Proof]{\par\medskip\noindent\textit{#1}\quad}{\par\medskip}
\newcommand{\qed}{\hfill $\Box$\par}

\allowdisplaybreaks[4]

\def\ar{\!\!\!&}
\def\mbb{\mathbb}

\def\beqlb{\begin{eqnarray}}\def\eeqlb{\end{eqnarray}}
\def\beqnn{\begin{eqnarray*}}\def\eeqnn{\end{eqnarray*}}

\renewcommand{\d}{\mathrm{d}}\newcommand{\e}{\mathrm{e}}

\title{Yaglom limits of continuous-state branching processes in Brownian random environment%
\thanks{Supported by the National Natural Science Foundation of China (NSFC) (Grants 12271029, 12301165, 12471135), NSERC (RGPIN-2021-04100) and the Young and Middle-aged Faculty Overseas Research and Training Program of East China University of Science and Technology (ECUST).}}

\author{
Pei-Sen Li\\
School of Mathematics and Statistics, Beijing Institute of Technology, Beijing 100872, China\\
\texttt{peisenli@bit.edu.cn}
\and
Xiangqi Zheng\thanks{Corresponding author.}\\
School of Mathematics, East China University of Science and Technology, Shanghai 201206, China\\
\texttt{zhengxq@ecust.edu.cn}
\and
Xiaowen Zhou\\
Department of Mathematics and Statistics, Concordia University, 1455\\ De Maisonneuve Blvd. W., Montreal, Canada\\
\texttt{xiaowen.zhou@concordia.ca}
}

\date{}

\begin{document}

\maketitle

\begin{abstract}
In this paper, we investigate the asymptotic behavior of continuous-state branching processes in a Brownian random environment (CBBRE) conditioned on non-extinction. For the subcritical case, we prove the existence of the Yaglom limit and derive an explicit representation of its Laplace transform using Kummer confluent hypergeometric functions. Notably, we demonstrate that the Yaglom limit is strictly independent of the initial state of the process across all three subcritical regimes: weakly, intermediately, and strongly subcritical.
\end{abstract}



\section{Introduction}

Galton-Watson processes in random environments (GWREs for short), introduced by Smith and Wilkinson \cite{smw}, are stochastic models that describe the evolution of populations. In the classical Galton-Watson processes, the reproduction law of the population is deterministic. In contrast, GWREs incorporate a stochastic reproduction law. Let $(f, f_1,f_2,\dots)$ be a sequence of i.i.d.\ random probability generating functions on $\mathbb{N}$, which serves as a random environment. In a GWRE $\{X_n: n=0, 1, 2,\dots\}$ induced by
 $(f, f_1,f_2,\dots)$, it is assumed that, conditioned on the environment $( f_1,f_2,\dots)$, the particles reproduce independently of each other and the number of offspring in generation $n$ has probability generating function $f_n$. Then the process conditioned on  $(f_1,f_2,\dots)$ can be defined recursively as follows:
\beqnn
X_{n+1}=\sum^{X_n}_{i=1}  \xi_{n, i}, \qquad n\in \mbb{N},
\eeqnn
where $\{\xi_{n ,i}\}$ is a sequence of independent random variables such that
\beqnn
\mbb{E}[s^{\xi_{n, i}}|(f_1,f_2,\dots,f_n)]=f_n(s),\qquad\mbox{for all}~ s\in(0,1).
\eeqnn

GWREs can be classified into three categories: supercritical when $\mbb{E}[\log (f'(1))]>0$, critical when $\mbb{E}[\log (f'(1))]=0$, and subcritical when $\mbb{E}[\log (f'(1))]<0$. In the subcritical case, the population goes extinct almost surely, that is,
$$ \mbb{P}(\exists n>0: X_n=0)=1. $$
Moreover, subcritical GWREs can be further divided into three distinct regimes: strongly subcritical, intermediately subcritical, and weakly subcritical. The rate at which the survival probability $\mbb{P}(X_n>0)$ decays to $0$ as $n\to\infty$ varies significantly across these regimes; for further details, see Geiger et al.\ \cite{gkv} and Bansaye \cite{bansaye}.

A closely related problem is to study the Yaglom limit starting from $k$ particles:
$$ \lim_{n\to\infty} \mbb{P}(X_n=i \mid X_0=k, X_n>0),\quad i\ge 1. $$
In his seminal work, Bansaye \cite{bansaye} derives an implicit equation that characterizes the probability generating function of the Yaglom limit. His study further shows that, in both the intermediately and strongly subcritical regimes, the Yaglom limit does not depend on the initial value. However, for the weakly subcritical regime, the dependence of the Yaglom limit on its initial value remains an open question; see \cite[Theorem 7]{bansaye}.

Motivated by this open problem in the discrete setting, we turn our attention to continuous-state analogs, which offer greater computational tractability and serve as natural reference systems for GWREs. Historically, Kurtz \cite{kurtz} developed a diffusion approximation for GWREs through an appropriate scaling transformation, leading to branching diffusions in random environments (BDREs). Subsequently, B\"{o}inghoff and Hutzenthaler \cite{BH12} characterized the exact asymptotic behavior of the survival probability for BDREs. Building on this framework, Palau and Pardo \cite{pp2} introduced continuous-state branching processes in a Brownian random environment (CBBREs) by incorporating jump structures into the aforementioned model. (For studies of branching processes in more general environments, we also refer to He et al.\ \cite{hlx}, Li and Xu \cite{lixu}, Palau and Pardo \cite{pp}, and Ma and Zhou \cite{mz}.) In this paper, we establish an explicit analytical representation of the Yaglom limit for subcritical CBBREs. Crucially, we demonstrate that this limit is strictly independent of the initial state across all subcritical regimes, thereby providing a continuous-time resolution to this dependency question.

To achieve explicit analytical results, in this work, we will focus on CBBREs with stable branching mechanism defined by
\beqlb\label{bm}\psi(\lambda)=-\alpha\lambda+c\lambda^{\beta+1},\quad \lambda\ge 0,\eeqlb
where  $c>0,\alpha\in\mathbb R, \beta\in(0, 1].$

Let $(B^{(e)}_t)_{t\ge 0}$ and $(B^{(b)}_t)_{t\ge 0}$ be mutually independent Brownian motions. For $\beta\in(0,1)$, let $\tilde N(\d s,\d z,\d u)$ be a compensated Poisson random measure with intensity
$\frac{c\beta(\beta+1)\mathbf 1_{\{z>0\}}}{\Gamma(1-\beta)z^{2+\beta}}\d s\d z\d u.$ Let $\sigma\in(0,\infty).$ When $\beta\in (0, 1)$, the CBBRE with stable branching mechanism $\psi$ can be constructed as the unique strong solution to
\begin{equation}\label{stablesde}
 Z_t=Z_0+\alpha\int_0^tZ_s\d s+\sigma \int_0^t Z_s \d B^{(e)}_s+\int_0^t\int_0^\infty\int_0^{Z_{s-}}z\tilde N(\d s,\d z,\d u).
\end{equation}
When $\beta=1,$ our model coincides with BDRE, which can be constructed as the unique strong solution to
\begin{equation}\label{browniansde}
	Z_t=Z_0+\alpha\int_0^tZ_s\d s+\sigma \int_0^t Z_s \d B^{(e)}_s+\int_0^t\sqrt{2c Z_s} \d B^{(b)}_s.
\end{equation}

Let $ \mathbb P^z\left(Z_t \in\cdot\right)=\mathbb P\left(Z_t \in\cdot\vert Z_0=z\right)$,  and let $\mathbb E^z$ be the corresponding expectation.
Denote ${\bf m}=\alpha-\frac{\sigma^2}{2}.$  The long-term behavior of survival probability follows a three-phase classification analogous to classical branching processes:  in the supercritical case (${\bf m} > 0$), survival probability remains strictly positive; in the critical case (${\bf m} = 0$), survival probability decays to zero at a polynomial rate;  while in the subcritical case (${\bf m} < 0$), survival probability exhibits exponential decay to zero. Moreover, \cite[Theorem 3]{pp2} establishes a finer subdivision of the subcritical case into three distinct regimes:
\begin{itemize}
\item if  ${\bf m}\in(-\sigma^2,0)$, then the process is called {\it weakly subcritical}, and
\beqnn\lim _{t \rightarrow \infty} {t}^{3/2} \e^{\frac{{\bf m}^2}{2 \sigma^2} t} \mathbb{P}^z\left(Z_t>0\right)  = \frac{8}{\beta^3\sigma^3} \int_0^{\infty}(1-\e^{-z(\frac{\beta\sigma^2}{2c}a)^{1/\beta}})\phi(a) \d a>0;\eeqnn
\item  if ${\bf m}=-\sigma^2$, then the process is called {\it intermediately subcritical}, and
\beqnn\lim _{t \rightarrow \infty} \sqrt{t} \e^{\frac{\sigma^2}{2} t} \mathbb{P}^z\left(Z_t>0\right)  = z \frac{\sqrt{2}}{\sqrt{\pi} \beta\sigma}\Gamma\left(\frac{1}{\beta}\right)\left(\frac{\beta\sigma^2}{2c}\right)^{1/\beta}>0;\eeqnn
\item if ${\bf m}<-\sigma^2$, then the process is called {\it strongly subcritical}, and
\beqnn\lim _{t \rightarrow \infty} \e^{-\left({\bf m}+\frac{\sigma^2}{2}\right) t} \mathbb{P}^z\left(Z_t>0\right)  = z\left(\frac{\beta\sigma^2}{2c}\right)^{1/\beta}\frac{\Gamma\left(\eta-\frac{1}{\beta}\right)}{\Gamma\left(\eta-\frac{2}{\beta}\right)}>0.\eeqnn
\end{itemize}
Here, $\eta=-\frac{2 {\bf m}}{\beta\sigma^2}$ and the function $\phi:(0, \infty) \rightarrow(0, \infty)$ is defined as
\begin{equation*}
\phi(a)=\int_0^{\infty} \int_0^{\infty} \frac{1}{\sqrt{2} \pi} \Gamma\left(\frac{\eta+2}{2}\right) \e^{-a} a^{-\eta / 2} u^{(\eta-1) / 2} \e^{-u} \frac{\sinh (\xi) \cosh (\xi) \xi}{\left(u+a(\cosh (\xi))^2\right)^{\frac{\eta+2}{2}}} \d \xi \d u.
\end{equation*}

The Yaglom limit of the CBBRE with initial value $z>0$ is defined by
\beqlb\label{yaglom}
\lim_{t\to\infty} \mbb{P}^z(Z_t\in\cdot|Z_t>0)
\eeqlb
In the sense of weak convergence of probability measures.  Define
\begin{equation}\label{yglmt}
\mathcal L_t(z,\lambda)=\mathbb E^z\left[\e^{-\lambda Z_t} \mid Z_t>0\right]=1-\frac{\mathbb E^z\left[1-\e^{-\lambda Z_t} \right]}{\mathbb P^z(Z_t>0)},\quad z>0,\,\,\lambda>0.\end{equation}
Then the associated Laplace transform of the Yaglom limit  is given by
\begin{equation*}
\mathcal L(z,\lambda)=\lim _{t \rightarrow \infty}\mathcal L_t(z,\lambda).\end{equation*}

We now present the main result of this paper.
\begin{theorem}\label{mainthm}
For any subcritical CBBRE with branching mechanism defined by \eqref{bm}, the Yaglom limit exists and is independent of the initial value $z$. Moreover, its Laplace transform is given by {\small\beqlb\label{mainthm0327}
\mathcal L(z,\lambda)=1-\left(\frac{\gamma}{2}\right)^{\frac{1}{\beta}\left(\Theta-\frac{2{\bf m}}{\sigma^2}\right)}U\left(\frac{1}{\beta}\left(\Theta-\frac{2{\bf m}}{\sigma^2}\right),1+\frac2\beta\left(\Theta-\frac{{\bf m}}{\sigma^2}\right),\frac{\gamma\lambda^{\beta}}{2}\right)\lambda^{\Theta-\frac{2{\bf m}}{\sigma^2}},\quad \lambda, z>0,
    \eeqlb}
	where $\Theta=\frac{{\bf m}}{\sigma^2}\vee(-1),$ $\gamma=\frac{4c}{\beta\sigma^2}$ and $U$ is the Kummer function defined by
\[U(a,b,r)=\frac{1}{\Gamma(a)}\int_0^\infty \e^{-rt}t^{a-1}(1+t)^{b-a-1}\d t,\qquad a>0,b\in\mathbb R,r>0.\]
\end{theorem}

\begin{remark}
		When $\sigma=0,$ the model degenerates into a continuous-state branching process without environment. In this situation,
		\beqnn
		\mathcal L(z,\lambda)\ar=\ar 1-\lim _{t \rightarrow \infty}\mathbb E\left[1-\e^{-\lambda Z_t} \mid Z_t>0\right]\cr
		\ar=\ar 1-\lim _{t \rightarrow \infty}\frac{1-\exp\{-z[{\lambda}^{-\beta}\e^{-{\bf m}\beta t}+\frac{c}{{\bf m}}(1-\e^{-{\bf m}\beta t})]^{-\frac{1}{\beta}}\}}{1-\exp\{-z[\frac{c}{{\bf m}}(1-\e^{-{\bf m}\beta t})]^{-\frac{1}{\beta}}\}}\cr
		\ar=\ar{1-\big(1-\frac{\alpha}{c}\lambda^{-\beta}\big)^{-\frac{1}{\beta}}.}
		\eeqnn
This result is consistent with \cite[Theorem 3.1]{lambert}.
\end{remark}

\begin{remark}
The result of Theorem \ref{mainthm} is consistent with \cite[Theorem 7]{bansaye}. As for GWREs, the Yaglom limit is also independent of the initial state in both the intermediately and strongly subcritical regimes. Although for weakly subcritical GWREs, whether the Yaglom limit depends on the initial value remains an open problem, Theorem \ref{mainthm} suggests that independence might still hold for this discrete model. However, the proof for the discrete model will, of course, be more involved.
\end{remark}

 \begin{remark}
In this work, we focus  on subcritical processes, as the limit \eqref{yaglom} becomes degenerate for critical or supercritical cases. In the critical case, the limit only makes sense after a proper scaling of the process. For a classical critical Galton-Watson process $\{X_n:n=0,1,2,\cdots\},$ Yaglom \cite{yaglom} established the existence of the conditional limit $$\lim_{n\to\infty} \mbb{P}\bigg(\frac{X_n}{n}\in\cdot|X_n>0\bigg)$$ under a moment condition. The analog of this result for continuous-state branching processes was given in Li \cite{lialone}. However, for critical branching processes in random environments, finding the proper scaling  becomes  more complex. To the best of our knowledge, the Yaglom limit for critical branching processes in random environments remains unexplored in the literature, both for discrete and continuous cases.
\end{remark}

\section{Proofs}

Recall
 \(\eta = -\frac{2 {\bf m}}{\beta\sigma^2}.\) It follows that
\begin{align*}
\eta &< \frac2\beta \quad \text{in the weakly subcritical regime } ({\bf m}\in(-\sigma^2,0)),\\
\eta &= \frac2\beta \quad \text{in the intermediately subcritical regime } ({\bf m}=-\sigma^2),\\
\eta &> \frac2\beta \quad \text{in the strongly subcritical regime } ({\bf m}\in(-\infty,-\sigma^2)).
\end{align*}
Define the exponential functional of Brownian motion as follows:
$$ A_t^{(\eta)}=\int_0^t \exp \left\{2\left(\eta s+B^{(e)}_s\right)\right\} \d s.$$
 The auxiliary process $ A_t^{(\eta)}$ plays a central role in the expression of the Yaglom limit and survival probability. Recall that $\gamma=\frac{4c}{\beta\sigma^2}.$ Define for $t>0,\lambda>0,$
 \beqlb\label{0412xi}
\xi_t(\lambda)=\left[\gamma A_{\frac{\sigma^2\beta^2}{4}t}^{(\eta)}+{\lambda}^{-\beta} \exp \left\{2\left(\frac{\sigma^2\beta^2}{4}\eta t+B^{(e)}_{\frac{\sigma^2\beta^2}{4}t}\right)\right\}\right]^{-\frac{1}{\beta}},\quad\zeta_t=\left[\gamma A_{\frac{\sigma^2\beta^2}{4}t}^{(\eta)}\right]^{-\frac{1}{\beta}}.
 \eeqlb

 The following theorem establishes a duality between the Laplace transform of the solution to SDE \eqref{stablesde} or \eqref{browniansde} and the Laplace transform of a stochastic process constructed from Brownian motion $(B_t^{(e)})_{t\ge 0}$ and its exponential functional $(A_t^{(\eta)})_{t\ge 0}$. The former is implicit, while the latter is explicit and has many existing results that can be exploited. Consequently, the Yaglom limit problem reduces to the problem of the ratio of the Laplace transforms of the two processes $(\xi_t(\lambda))_{t\ge 0}$ and $(\zeta_t)_{t\ge 0}$:
 \[\mathcal L(z,\lambda)=1-\lim _{t \rightarrow \infty}\frac{\mathbb E\left[1-\e^{-z \xi_t(\lambda)} \right]}{\mathbb E\left[1-\e^{-z \zeta_t} \right]},\quad z>0,\,\,\lambda>0.\]
\begin{theorem}\label{firstlem}
For any $z>0,\lambda>0,$ we have
\begin{equation}\label{1'}
\mathbb E^{z}\left[\exp \left(-\lambda Z_t \right)\right]=\mathbb E\left[\exp \left(-z \xi_t(\lambda) \right)\right],
\end{equation}
and thus
\begin{equation}\label{2}
 \mathbb{P}^z\left(Z_t>0\right)=1-\mathbb E\left[\exp \left(-z \zeta_t\right)\right].
\end{equation}
\end{theorem}
\begin{proof}
Let
$$ H_t=\sigma B^{(e)}_t-\frac{\sigma^2}{2}t,\qquad H_t^{(0)}=\sigma B^{(e)}_t+{\bf m}t.$$
According to \cite[(16)]{pp2}, for all $t, z, \lambda \in(0, \infty),$ we have
\begin{equation*}
\mathbb E^{z}\left[\exp \left(-\lambda Z_t \e^{-H^{(0)}_t}\right)\mid H\right]=\exp \left\{-z\left[c\beta\int_0^t  \exp \left(-\beta H^{(0)}_s\right)\d s+{\lambda}^{-\beta} \right]^{-\frac{1}{\beta}}\right\},
\end{equation*}
which implies the following identity
\begin{equation*}
\mathbb E^{z}\left[\exp \left(-\lambda Z_t \right)\right]=\mathbb E\exp \left\{-z\left[c\beta\int_0^t  \exp \left(-\beta H_s^{(0)}\right)\d s+{\lambda}^{-\beta}\e^{-\beta H_t^{(0)}} \right]^{-\frac{1}{\beta}}\right\}.
\end{equation*}

Using the scaling property of Brownian motion, see  \cite[pp.54]{hbbm},
\begin{equation}\label{revise1}
(-\beta  H_s^{(0)})_{0\le s\le t}=(-{\bf m} \beta s-\sigma\beta B^{(e)}_s)_{0\le s\le t} \stackrel{d}{=} \left(\eta \frac{\sigma^2\beta^2}{2} s+2 B^{(e)}_{\frac{\sigma^2\beta^2}{4} s}\right)_{0\le s\leq t}.
\end{equation}
Hence,
\beqlb\label{112000}
\int_0^t  \exp \left(-\beta H_s^{(0)}\right) \d s\ar\stackrel{d}{=}\ar\int_0^t \exp \left(2 \eta \frac{\sigma^2\beta^2}{4} s+2 B^{(e)}_{\frac{\sigma^2\beta^2}{4} s}\right) \d s\cr\cr
\ar=\ar \int_0^{\frac{\beta^2\sigma^2}{4}t} \exp \left(2 \eta u+2 B^{(e)}_u\right) \frac{4}{\sigma^2\beta^2}\d u\cr\cr
\ar=\ar \frac{4}{\beta^2\sigma^2}\int_0^{\frac{\beta^2\sigma^2}{4}t} \exp \left(2 \eta u+2 B^{(e)}_u\right) \d u.
\eeqlb
By the process-level identity in distribution in \eqref{revise1}, the pair
\[
\left(
\int_0^t \e^{-\beta H_s^{(0)}}\,\d s,\,
{-\beta H_t^{(0)}}
\right)
\]
has the same distribution as
\[
\left(
\frac{4}{\sigma^2\beta^2}A^{(\eta)}_{\frac{\sigma^2\beta^2}{4}t},\,
2\left(\eta \frac{\sigma^2\beta^2}{4}t+B^{(e)}_{\frac{\sigma^2\beta^2}{4}t}\right)
\right).
\]
Therefore,
\beqnn
\ar\ar\left[c\beta\int_0^t \exp \left(-\beta H_s^{(0)}\right)\d s + \lambda^{-\beta}\e^{-\beta H_t^{(0)}} \right]^{-\frac{1}{\beta}} \\
\ar\ar\qquad\stackrel{d}{=} \left[\gamma A_{\frac{\sigma^2\beta^2}{4}t}^{(\eta)} + \lambda^{-\beta} \exp \left\{2\left(\frac{\sigma^2\beta^2}{4}\eta t + B^{(e)}_{\frac{\sigma^2\beta^2}{4}t}\right)\right\}\right]^{-\frac{1}{\beta}}=\xi_t(\lambda),
\eeqnn
and \eqref{1'} follows. Letting $\lambda\to\infty$ in the above, we obtain \eqref{2}.
\qed
\end{proof}
For the weakly subcritical case, we take advantage of the joint density function of $A_t^{(\eta)}$ and $\left(\eta t+B^{(e)}_t\right)$ to derive integral representations for
$\mathbb E\left[1-\e^{-z \xi_t(\lambda)} \right]$ and \(\mathbb E\left[1-\e^{-z \zeta_t} \right].\) We then identify their asymptotic limits by passing the limit inside the integrals. The independence of the Laplace transform of the Yaglom limit is proved via a change of measure.

For the intermediately and strongly subcritical regimes, the  estimates that justify the interchange of limit and integration used in the weakly subcritical case are no longer available. Thus, it is difficult to obtain the explicit expressions of $\lim_{t\rightarrow\infty}\mathbb E\left[1-\e^{-z \xi_t(\lambda)} \right]$ and \(\lim_{t\rightarrow\infty}\mathbb E\left[1-\e^{-z \zeta_t} \right]\) directly from the density. We therefore use Lemma \ref{fenzifenmu}  to reduce the problem to the asymptotic analysis of suitable first moments of $\xi_t(\lambda)$ and $\zeta_t$, after showing that the corresponding second-order terms are negligible under the chosen normalization. The resulting first-moment asymptotics are then derived from Lemma \ref{bh12} in the intermediately subcritical case and from Dufresne's identity in the strongly subcritical case. It turns out that the first moments are both linear functions of the initial value $z,$ therefore the Laplace transform of the Yaglom limits in these two regimes are also independent of the initial value. This method is not suitable for the weakly subcritical case, since normalized second-order terms are not negligible in that regime.

\subsection{The weakly subcritical regime}
In the weakly subcritical regime, where $\eta<\frac2\beta,$ we will give the explicit expression of the Laplace transform of the Yaglom limit and show that it is independent of the initial value. We need the following result.
\begin{proposition}[Matsumoto and Yor \cite{my03}]\label{jtpdf}
For any $t>0,$ the joint probability density function of $A_t^{(\eta)}$ and $\left(\eta t+B^{(e)}_t\right)$ is given by
\beqnn
\ar\ar\mathbb P\left(A_t^{(\eta)} \in \d y, \eta t+B^{(e)}_t \in \d x\right) =\e^{\eta x-\frac{1}{2} \eta^2t} \frac{1}{y} \exp \left(-\frac{1+\e^{2 x}}{2 y}\right) \theta\left(\frac{\e^x}{y}, t\right) \d y \d x,
\eeqnn
where $$\theta(r, t)=\frac{r}{\sqrt{2 \pi^3 t}} \int_0^{\infty} \e^{\frac{\pi^2-u^2}{2 t}} \e^{-r \cosh u}\sinh (u) \sin \left(\frac{\pi u}{t}\right) \d u,\quad r>0,\,t>0.$$
\end{proposition}
Our ultimate goal in this regime is to establish the exact asymptotic behavior of expectation $F(t) := \mbb{E}^z[1-\e^{-\lambda Z_t}]$ and survival probability $\mbb{P}^z(Z_t>0)$ as $t \to \infty$. Anticipating this temporal limit, we  define a  time-independent kernel $G(y,v)$ for $y>0$ and $v>0$:
\beqlb\label{G_def}
G(y,v)=2^{\frac{\eta}{2}-1}y^{\frac{\eta}{2}-1}\exp\left(-\frac{1}{2y}\right)\,v^{\frac{\eta}{2}-1} \e^{-v}K_0\left(\sqrt{\frac{2v}{y}}\right),
\eeqlb
where $K_0$ is the modified Bessel function of the second kind. For subsequent calculations, we recall its integral representation (see \cite[Chapter 10, \S 10.25 and \S 10.32]{NISTHandbook}):
\beqlb\label{k0integral}
K_0(a)=a\int_0^\infty u\,\e^{-a\cosh u}\sinh u\,\d u = \frac{1}{2}\int_0^\infty t^{-1}\exp\left(-t-\frac{a^2}{4t}\right)\d t, \qquad a>0.
\eeqlb

For any  $z>0, \lambda>0$, we define the functions $\Psi(z)$ and $\Phi(z, \lambda)$ as follows:

\beqnn\Phi(z,\lambda)\ar=\ar \int_0^\infty \int_0^\infty  \bigl(1-\e^{-z \left(\gamma  y+2yv{\lambda}^{-\beta} \right)^{-\frac1\beta}}\bigr) G(y,v) \, \d v \, \d y,
	\eeqnn
    and
	\beqnn\Psi(z)\ar=\ar \int_0^\infty\int_0^\infty \bigl(1-\e^{-z (\gamma y)^{-\frac{1}{\beta}} }\bigr)G(y,v) \, \d v \, \d y.
	\eeqnn

We will rigorously establish that the appropriately scaled expectation $\mbb{E}^z[1-\e^{-\lambda Z_t}]$ and the survival probability $\mbb{P}^z(Z_t>0)$ converge to constant multiples of $\Phi(z, \lambda)$ and $\Psi(z)$, respectively. The Laplace transform of the Yaglom limit can thus be rewritten as
\[\mathcal L(z,\lambda)=1-\frac{\Phi(z,\lambda)}{\Psi(z)}.\]

To determine whether $\mathcal L(z,\lambda)$ depends on the initial value $z$, a natural starting point is to take the Taylor expansions in both  the numerator and the denominator.
Then we compare the corresponding  coefficients. For example, expanding the integrand of the denominator $\Psi(z)$ yields
\beqlb\label{eq041100}
\Psi(z)=\int_0^\infty\int_0^\infty
\sum_{n=1}^{\infty}\frac{(-1)^{n+1}}{n!}z^n(\gamma y)^{-n/\beta}G(y,v) \, \d v \, \d y.
\eeqlb
If the interchange of summation and integration were justified, $\Psi(z)$ could be expressed as an infinite order polynomial in $z$, where the coefficient of $z^n$ is proportional to the integral $$\int_0^\infty\int_0^\infty (\gamma y)^{-n/\beta}G(y,v)\,\d v\,\d y.$$ Analogously, a formal expansion of the numerator $\Phi(z, \lambda)$ produces a coefficient $z^n$ proportional to the integral
\beqlb\label{formal_phi_coeff}
\int_0^\infty\int_0^\infty \left(\gamma y+2yv{\lambda}^{-\beta} \right)^{-n/\beta}G(y,v)\,\d v\,\d y.
\eeqlb
Letting  $p=\frac n\beta$ in the following Lemma \ref{keyprop0}, we find that the ratios between the corresponding coefficients in the expansions for $\Phi$ and $\Psi$ remain  the same for all $n$. This common ratio would immediately imply that $\Phi(z, \lambda)/\Psi(z)$ is independent of $z.$

\begin{lemma}\label{keyprop0}
For any $p>\frac\eta2,$  we have
\beqlb\label{psimoment}
\int_0^\infty\int_0^\infty (\gamma y)^{-p}G(y,v)\,\d v\,\d y =2^{p - 2} \gamma^{- p}  \frac{\left[ \Gamma \left( \frac\eta2 \right)\Gamma\left( p-\frac\eta2  \right)  \right]^2}{\Gamma\left( p \right)},
\eeqlb
and
{\small\beqlb\label{phimoment}
\int_0^\infty\int_0^\infty (\gamma y+2yv{\lambda}^{-\beta} )^{-p}G(y,v)\,\d v\,\d y
= 2^{p - 2} \gamma^{- p} \frac{ \left[\Gamma \left( \frac\eta2 \right)\Gamma \left( p - \frac\eta2 \right)\right]^2}{\Gamma \left( p \right)} \left( \frac{\gamma  \lambda^\beta}{2} \right)^{\frac\eta2} U \left( \frac\eta2, 1, \frac{\gamma  \lambda^\beta}{2} \right).
\eeqlb}
\end{lemma}

However, the formal argument discussed before Lemma \ref{keyprop0} cannot be directly justified because the resulting series do not converge absolutely, rendering the interchange of summation and integration invalid. Instead of dealing with divergent power series, we introduce two finite, light-tailed reference measures on $\mbb{R}$, denoted by $\nu_\Psi$ and $\nu_\Phi$.

In this setting, the formal coefficients of $z^n$ in the Taylor expansions are directly transformed into the moment generating functions of these new measures. For example, for $\Psi(z)$, the equivalence is given by:
$$ \int_0^\infty\int_0^\infty (\gamma y)^{-n/\beta}G(y,v)\,\d v\,\d y = \int_{-\infty}^{\infty}\e^{q x}\nu_\Psi(\d x), \quad \text{with} \quad q = n - \beta\eta. $$
Specifically, we define the measures $\nu_\Psi$ and $\nu_\Phi$ via their Radon-Nikodym derivatives.
\beqlb\label{nu_psi_def}
\frac{\nu_\Psi(\d x)}{\d x} = \beta \gamma^{-1} \e^{-\beta (1-\eta) x} \int_0^{\infty} G\bigl(\gamma^{-1} \e^{-\beta x}, v\bigr) \, \d v,
\eeqlb
and
\beqlb\label{nu_phi_def}
\frac{\nu_\Phi(\d x)}{\d x} = \frac{\beta^2 \lambda^\beta}{2} \int_0^{\gamma^{-1/\beta}} \e^{-\beta(1-\eta)x} w^{-1} G\left(\e^{-\beta x} w^{\beta},\frac{\lambda^\beta}{2}(w^{-\beta}-\gamma)\right) \,\d w.
\eeqlb
Using these measures, the original functions $\Psi(z)$ and $\Phi(z, \lambda)$ can be rewritten as integrals of an identical  function against $\nu_\Psi$ and $\nu_\Phi$:
\beqlb\label{phi_psi_def}
\Psi(z) = \int_{-\infty}^\infty \bigl(1 - \e^{-z \e^x}\bigr) \e^{-\beta\eta x}\,\nu_\Psi(\d x), \quad \Phi(z, \lambda) = \int_{-\infty}^\infty \bigl(1 - \e^{-z \e^x}\bigr) \e^{-\beta\eta x}\,\nu_\Phi(\d x).
\eeqlb

The following Proposition \ref{keyprop} is key in proving that the Yaglom limit does not depend on the initial value. In this proposition, we establish a strict linear proportionality between the moment generating functions of the two measures. Because $\Phi(z, \lambda)$ and $\Psi(z)$ share the exact same integrand in this new representation, the proportionality of the measures directly translates into the functions themselves. Consequently, the ratio $\Phi(z,\lambda)/\Psi(z)$ cancels out the initial state $z$, rigorously proving that the Laplace transform of the Yaglom limit is independent of $z$.
\begin{proposition}\label{keyprop}
For any $q > -\frac{\beta\eta}{2}$, the moment generating functions of $\nu_\Phi$ and $\nu_\Psi$ are given by:
\beqlb\label{phimoment1}
\int_{-\infty}^\infty \e^{q x} \, \nu_{\Phi} (\d x) = 2^{\frac{q}{\beta} + \eta - 2} \gamma^{- (\frac{q}{\beta} + \eta)} \frac{ \left[\Gamma \left( \frac\eta2 \right)\Gamma \left( \frac{q}{\beta} + \frac\eta2 \right)\right]^2}{\Gamma \left( \frac{q}{\beta} + \eta \right)} \left( \frac{\gamma \lambda^\beta}{2} \right)^{\frac\eta2} U \left( \frac\eta2, 1, \frac{\gamma  \lambda^\beta}{2} \right),
\eeqlb
and
\beqlb\label{psimoment1}
\int_{-\infty}^\infty \e^{q x} \, \nu_{\Psi} (\d x) = 2^{\frac{q}{\beta} + \eta - 2} \gamma^{- (\frac{q}{\beta} + \eta)} \frac{ \left[\Gamma \left( \frac\eta2 \right)\Gamma \left( \frac{q}{\beta} + \frac\eta2 \right)\right]^2}{\Gamma \left( \frac{q}{\beta} + \eta \right)}.
\eeqlb
\end{proposition}

\noindent{\it Proof of Lemma \ref{keyprop0}.}
We first prove \eqref{phimoment}. It is easy to verify that
\beqlb\label{phi0313}
\ar\ar\int_0^\infty\int_0^\infty  \left(\gamma y+2yv{\lambda}^{-\beta} \right)^{-p}G(y,v)\,\d v\,\d y
\cr\cr\ar=\ar2^{\frac\eta2-1}\int_0^\infty\int_0^\infty\left(\gamma  y+2yv{\lambda}^{-\beta} \right)^{-p}y^{\frac{\eta}{2}-1}\,v^{\frac{\eta}{2}-1} \e^{-\frac{1}{2y}-v}K_0\left(\sqrt{\frac{2v}{y}}\right)\,\d v\,\d y\cr\cr
\ar=:\ar2^{\frac\eta2-1}\int_0^\infty\left(\gamma  +2v{\lambda}^{-\beta} \right)^{-p}\,v^{\frac{\eta}{2}-1} \e^{-v}J\left(v,p-\frac\eta2\right)\,\d v.
\eeqlb
Substituting the integral representation of $K_0$ in \eqref{k0integral}, letting $r = 1/y$, and evaluating the resulting Gamma integral over $r$, we obtain
\beqlb\label{J_eval}
J(v,p-\frac\eta2) \ar=\ar \frac12\int_0^\infty t^{-1}\e^{-t} \left[ \int_0^\infty r^{p-\frac\eta2-1}\e^{-r\left(\frac{1}{2}+\frac{v}{2t}\right)}\,\d r \right] \d t \cr\cr
\ar=\ar 2^{p-\frac\eta2-1}\Gamma\left(p-\frac\eta2\right)\int_0^\infty t^{p-\frac\eta2-1}\e^{-t}(t+v)^{-p+\frac\eta2}\,\d t\cr\cr
\ar =\ar 2^{p-\frac\eta2-1}\Gamma\left(p-\frac\eta2\right)\int_0^\infty s^{p-\frac\eta2-1}(1+s)^{-p+\frac\eta2}\e^{-vs}\,\d s.
\eeqlb
Substituting \eqref{J_eval} into \eqref{phi0313}, we have
\beqlb\label{eq040902}\ar\ar\int_0^\infty\int_0^\infty  \left(\gamma y+2yv{\lambda}^{-\beta} \right)^{-p}G(y,v)\,\d v\,\d y\cr\cr
\ar=\ar 2^{p-2}\Gamma\left(p-\frac\eta2\right)\int_0^\infty\left(\gamma +2v{\lambda}^{-\beta} \right)^{-p}\,v^{\frac{\eta}{2}-1} \e^{-v}\int_0^\infty s^{p-\frac\eta2-1}(1+s)^{\frac\eta2-p}\e^{-vs}\,\d s\,\d v\cr\cr
\ar=\ar 2^{p-2}\Gamma\left(p-\frac\eta2\right)\int_0^\infty\,w^{\frac{\eta}{2}-1} \e^{-w}\int_0^\infty \left[\gamma (1+s) +2w{\lambda}^{-\beta} \right]^{-p}s^{p-\frac\eta2-1}\,\d s\,\d w,\eeqlb
where the last integral follows by a change of variable $w=(1+s)v.$ The inner integral on $s$ matches the generalized Beta integral identity $$\int_0^\infty x^{a-1}(A+Bx)^{-b}\d x = A^{a-b}B^{-a} \frac{\Gamma(a)\Gamma(b-a)}{\Gamma(b)}.$$ Setting the constants as $A = \gamma + 2w\lambda^{-\beta}$ and $B = \gamma$, it yields
\beqlb\label{eq441}
\int_0^\infty \left[\gamma(1+s) +2w{\lambda}^{-\beta} \right]^{-p}s^{p-\frac\eta2-1}\,\d s= \gamma^{-p+\frac\eta2}\left(\gamma +2w{\lambda}^{-\beta}\right)^{-\frac\eta2}\frac{\Gamma\left(p-\frac \eta2\right)\Gamma\left(\frac \eta2\right)}{\Gamma(p)}.\eeqlb
Substituting \eqref{eq441} into \eqref{eq040902} and let $w=\frac{\gamma\lambda^\beta}{2}t,$ we get
\beqnn
\ar\ar\int_0^\infty\int_0^\infty  \left(\gamma y+2yv{\lambda}^{-\beta} \right)^{-p}G(y,v)\,\d v\,\d y\cr\cr
\ar=\ar 2^{p-2}\gamma^{-p}\frac{\left[\Gamma\left(p-\frac \eta2\right)\right]^2\Gamma\left(\frac \eta2\right)}{\Gamma(p)}\int_0^\infty\,\left(\frac{\gamma\lambda^\beta}{2}\right)^{\frac\eta2}t^{\frac{\eta}{2}-1} \e^{-\frac{\gamma\lambda^\beta}{2}t} \left(1+t\right)^{-\frac\eta2}\d t\cr\cr
\ar=\ar 2^{p-2}\gamma^{-p}\frac{\left[\Gamma\left(p-\frac \eta2\right)\Gamma\left(\frac \eta2\right)\right]^2}{\Gamma(p)}\left(\frac{\gamma \lambda^\beta}{2}\right)^{\frac\eta2}U \left( \frac\eta2, 1, \frac{\gamma  \lambda^\beta}{2} \right).\eeqnn
Thus, we have proved  \eqref{phimoment}.

Since $\left(\gamma  y+2yv{\lambda}^{-\beta} \right)^{-p}$ increases in $\lambda$, applying the monotone convergence theorem to the above, we have
\beqnn
\ar\ar\int_0^\infty\int_0^\infty  \left(\gamma y \right)^{-p}G(y,v)\,\d v\,\d y\cr\cr
\ar=\ar\lim_{\lambda\rightarrow\infty} 2^{p-2}\gamma^{-p}\frac{\left[\Gamma\left(p-\frac \eta2\right)\Gamma\left(\frac \eta2\right)\right]^2}{\Gamma(p)}\left(\frac{\gamma \lambda^\beta}{2}\right)^{\frac\eta2}U \left( \frac\eta2, 1, \frac{\gamma  \lambda^\beta}{2} \right).\eeqnn
On the other hand, the classical asymptotic expansion of the Kummer function guarantees that $U(a, b, r) \sim r^{-a}$ as $r \to \infty$. Therefore,  $$\lim_{\lambda \to \infty}\left(\frac{\gamma \lambda^\beta}{2}\right)^{\frac\eta2}U \left( \frac\eta2, 1, \frac{\gamma  \lambda^\beta}{2} \right)=1.$$   Letting $\lambda\rightarrow\infty$ in \eqref{phimoment}, we obtain \eqref{psimoment}.
\qed

\medskip

\noindent{\it Proof of Proposition \ref{keyprop}.}
Let $p = \frac{q}{\beta} + \eta > \frac{\eta}{2}$. We first compute the moment generating function for $\nu_\Phi$.  By substituting the density \eqref{nu_phi_def} into the integral, we have
{\small\beqnn
\int_{-\infty}^\infty \e^{q x} \, \nu_{\Phi} (\d x) = \int_{-\infty}^\infty \int_0^{\gamma^{-1/\beta}} \e^{qx} \left[ \frac{\beta^2 \lambda^\beta}{2} \e^{-\beta(1-\eta)x} w^{-1} G\left(\e^{-\beta x} w^{\beta},\frac{\lambda^\beta}{2}(w^{-\beta}-\gamma)\right) \right] \d w \d x.
\eeqnn}
To evaluate this integral, we apply a change of variables $(x, w) \mapsto (y, v)$ given by
\beqnn
y = \e^{-\beta x} w^\beta \quad \text{and} \quad v = \frac{\lambda^\beta}{2}(w^{-\beta}-\gamma),
\eeqnn
and find its Jacobian determinant. The partial derivatives are as follows:
\beqnn
\frac{\partial y}{\partial x} = -\beta \e^{-\beta x} w^\beta, \quad \frac{\partial y}{\partial w} = \beta \e^{-\beta x} w^{\beta-1}, \quad \frac{\partial v}{\partial x} = 0, \quad \frac{\partial v}{\partial w} = -\frac{\beta \lambda^\beta}{2} w^{-\beta-1}.
\eeqnn
Then, the absolute value of the Jacobian is
\beqnn
\left| \frac{\partial(y,v)}{\partial(x,w)} \right| = \left| \left(-\beta \e^{-\beta x} w^\beta\right) \left(-\frac{\beta \lambda^\beta}{2} w^{-\beta-1}\right) - 0 \right| = \frac{\beta^2 \lambda^\beta}{2} \e^{-\beta x} w^{-1}.
\eeqnn
Thus,
\beqnn
\d y \d v = \left| \frac{\partial(y,v)}{\partial(x,w)} \right| \d x \d w = \frac{\beta^2 \lambda^\beta}{2} \e^{-\beta x} w^{-1} \d x \d w.
\eeqnn
Since $w^{-\beta} = \gamma + 2v\lambda^{-\beta}$ and $\e^{-\beta x} = y w^{-\beta}$, we have
\beqnn
\e^{q x} \left[ \frac{\beta^2 \lambda^\beta}{2} \e^{-\beta(1-\eta)x} w^{-1} \right] \d x \d w
\ar=\ar \e^{(q + \beta\eta)x} \left( \frac{\beta^2 \lambda^\beta}{2} \e^{-\beta x} w^{-1} \d x \d w \right) \cr\cr
\ar=\ar (\e^{-\beta x})^{-(q/\beta + \eta)} \d y \d v \cr\cr
\ar=\ar (y w^{-\beta})^{-p} \d y \d v \cr\cr
\ar=\ar\left[ y \left( \gamma + 2v\lambda^{-\beta} \right) \right]^{-p} \d y \d v.
\eeqnn
Substituting the explicit definition of $G(y,v)$ and collecting all terms involving $y$, the original expectation translates into a cleanly separated double integral.
{\small\beqlb\label{eq040701}
\int_{-\infty}^\infty \e^{q x} \, \nu_{\Phi} (\d x)
\ar =\ar \int_0^\infty\int_0^\infty (\gamma y+2yv{\lambda}^{-\beta})^{-\frac q\beta-\eta}G(y,v)\,\d v \,\d y \cr\cr
\ar=\ar2^{\frac{q + {\beta\eta}}{\beta} - 2} \gamma^{- \frac{q + {\beta\eta}}{\beta}} \frac{ \left[\Gamma \left( \frac\eta2 \right)\Gamma \left( \frac{q + {\beta\eta}}{\beta} - \frac\eta2 \right)\right]^2}{\Gamma \left( \frac{q + {\beta\eta}}{\beta} \right)} \left( \frac{\gamma \lambda^\beta}{2} \right)^{\frac\eta2} U \left( \frac\eta2, 1, \frac{\gamma  \lambda^\beta}{2} \right),
\eeqlb}
where the last equality follows from Lemma \ref{keyprop0}.

Since $\left(\gamma y+2yv{\lambda}^{-\beta} \right)^{-p} G(y,v)$ increases in $\lambda$, by the monotone convergence theorem, we obtain,
\beqnn
\lim_{\lambda \to \infty} \int_{-\infty}^\infty \e^{q x} \, \nu_{\Phi} (\d x) \ar=\ar \lim_{\lambda \to \infty} \int_0^\infty\int_0^\infty \left(\gamma y+2yv{\lambda}^{-\beta} \right)^{-p} G(y,v)\,\d v\,\d y \cr\cr
\ar=\ar \int_0^\infty\int_0^\infty (\gamma y)^{-p} G(y,v)\,\d v\,\d y.
\eeqnn
 Combining this result with \eqref{psimoment}, we get,
\beqnn
\int_{-\infty}^\infty \e^{q x} \, \nu_{\Psi} (\d x) = 2^{p-2}\gamma^{-p}\frac{\left[\Gamma\left(p-\frac \eta2\right)\Gamma\left(\frac \eta2\right)\right]^2}{\Gamma(p)}.
\eeqnn
This establishes \eqref{psimoment1} and completes the proof.
\qed

\begin{corollary}\label{cor0409} Suppose that $\eta<\frac2\beta.$ For any $\lambda,z>0,$
\beqnn\frac{\Phi(z,\lambda)}{\Psi(z)}=\left(\frac{\gamma }{2}\right)^{\frac{\eta}{2}}U\left(\frac{\eta}{2},1,\frac{\gamma \lambda^{\beta}}{2}\right)\lambda^{\frac{\eta\beta}{2}}.\eeqnn
\end{corollary}

\begin{proof}
Since $\Psi(z)$ and $\Phi(z, \lambda)$ can be rewritten as integrals of a common integrand against the measures $\nu_{\Phi}$ and $\nu_{\Psi}$. It is sufficient to show that
\begin{equation}\label{measuregoal}
\nu_{\Phi}(\d x) = \left(\frac{\gamma \lambda^{\beta}}{2}\right)^{\frac{\eta}{2}}U\left(\frac{\eta}{2},1,\frac{\gamma \lambda^{\beta}}{2}\right) \nu_\Psi(\d x).
\end{equation}

By letting $q=0$ in the moment generating functions \eqref{phimoment1} and \eqref{psimoment1}, we obtain the total masses of the two measures:
\beqnn
\nu_\Phi(\mbb R) \ar=\ar 2^{\eta - 2} \gamma^{-\eta} \frac{ \left[\Gamma \left( \frac\eta2 \right)\right]^4}{\Gamma \left( \eta \right)} \left( \frac{\gamma \lambda^\beta}{2} \right)^{\frac\eta2} U \left( \frac\eta2, 1, \frac{\gamma \lambda^\beta}{2} \right), \cr\cr
\nu_\Psi(\mbb R) \ar=\ar 2^{\eta - 2} \gamma^{-\eta} \frac{ \left[\Gamma \left( \frac\eta2 \right)\right]^4}{\Gamma \left( \eta \right)}.
\eeqnn
Comparing \eqref{phimoment1} and \eqref{psimoment1} with these total masses, it is evident that for any $q \in \left(-\frac{\beta\eta}{2}, \infty\right)$, the moment generating functions of the normalized probability measures $\nu_{\Phi}/\nu_{\Phi}(\mbb R)$ and $\nu_{\Psi}/\nu_{\Psi}(\mbb R)$ are identical.

Since a probability measure is uniquely determined if its moment generating function exists in a neighborhood of zero (see, e.g., \cite[Theorem 30.1]{billingsley}), the normalized probability measures $\nu_{\Phi}/\nu_{\Phi}(\mbb R)$ and $\nu_{\Psi}/\nu_{\Psi}(\mbb R)$ are identical. Multiplying both sides by their respective total masses, and noting the exact ratio $\nu_\Phi(\mbb R) / \nu_\Psi(\mbb R) = \left(\frac{\gamma \lambda^\beta}{2}\right)^{\frac\eta2} U \left( \frac\eta2, 1, \frac{\gamma \lambda^\beta}{2} \right)$, directly yields \eqref{measuregoal}.

\qed
\end{proof}

\begin{lemma}\label{lem_Gt_asymp}
In the weakly subcritical regime $({\bf m}\in(-\sigma^2,0))$, let the time-dependent function $G_t(y, v)$ be defined as:
\beqlb\label{Gt_def_time}
G_t(y, v) := 2^{\frac{\eta}{2}-1} y^{\frac{\eta}{2}-1} v^{\frac{\eta}{2}-1} \e^{- \frac{{\bf m}^2}{2\sigma^2}t} \e^{-\frac{1}{2y} - v} \theta\left( \sqrt{\frac{2v}{y}}, \frac{\sigma^2\beta^2}{4}t \right),
\eeqlb
where $\theta(r, t)$ is defined in Proposition \ref{jtpdf}. Then, for any $y>0$ and $v>0$,
\beqnn
\lim _{t \rightarrow \infty} t^{3/2} \e^{\frac{{\bf m}^2}{2 \sigma^2} t} G_t(y, v) = \frac{4\sqrt 2}{\sigma^3\beta^3\sqrt{ \pi }} G(y, v),
\eeqnn
where $G(y,v)$ is the time-independent kernel defined in \eqref{G_def}. Furthermore, there exists a constant $C' > 0$ such that for all $t \ge 1$, the scaled function satisfies the uniform bound:
\beqnn
\left| t^{3/2} \e^{\frac{{\bf m}^2}{2 \sigma^2} t} G_t(y, v) \right| \le C' G(y,v).
\eeqnn
\end{lemma}
\begin{proof}
Let $\tau_t = \frac{\sigma^2\beta^2}{4}t$.  Multiplying $G_t(y, v)$ by $t^{3/2} \e^{\frac{{\bf m}^2}{2 \sigma^2} t}$, the exponential drift term $\e^{-\frac{{\bf m}^2}{2\sigma^2}t}$ cancels exactly:
\beqlb\label{26040901}
t^{3/2} \e^{\frac{{\bf m}^2}{2 \sigma^2} t} G_t(y, v) = 2^{\frac{\eta}{2}-1} y^{\frac{\eta}{2}-1} v^{\frac{\eta}{2}-1} \e^{-\frac{1}{2y} - v} \left[ t^{3/2} \theta\left( \sqrt{\frac{2v}{y}}, \tau_t \right) \right].
\eeqlb
To establish convergence and uniform bounds, we analyze the remaining component $t^{3/2} \theta(r, \tau_t)$ with $r = \sqrt{2v/y}$. By multiplying and dividing the integrand of $\theta(r, \tau_t)$ by $\frac{\pi u}{\tau_t}$, we extract $t^{3/2}$:
\beqnn
t^{3/2} \theta(r, \tau_t) = \frac{4\sqrt{2} r}{\sigma^3 \beta^3 \sqrt{\pi}} \int_0^{\infty} u \e^{\frac{\pi^2-u^2}{2 \tau_t}} \e^{-r \cosh u}\sinh (u) \left[ \frac{\sin\left(\frac{\pi u}{\tau_t}\right)}{\frac{\pi u}{\tau_t}} \right] \d u.
\eeqnn
To evaluate the limit, we note that $\big|\frac{\sin\epsilon}{\epsilon}\big| \le 1$ for all $\epsilon > 0$, and the exponential term satisfies $\e^{\frac{\pi^2 - u^2}{2\tau_t}} \le \e^{\frac{\pi^2}{2\tau_1}} := C$ for any $t \ge 1$. Consequently, the absolute value of the integrand is uniformly bounded by a time-independent integrable function:
\beqnn
\left| u \e^{\frac{\pi^2-u^2}{2 \tau_t}} \e^{-r \cosh u}\sinh (u) \left[ \frac{\sin\left(\frac{\pi u}{\tau_t}\right)}{\frac{\pi u}{\tau_t}} \right] \right| \le C u \e^{-r \cosh u}\sinh(u).
\eeqnn
Since $\lim_{\epsilon \to 0}\frac{\sin \epsilon}{\epsilon} = 1$ and $\lim_{t\to\infty}\e^{\frac{\pi^2-u^2}{2 \tau_t}}=1$, the dominated convergence theorem applies, allowing us to pass the limit inside the integral:
\beqnn
\lim_{t \to \infty} t^{3/2} \theta(r, \tau_t) = \frac{4\sqrt{2} r}{\sigma^3\beta^3\sqrt{\pi}} \int_0^\infty u \e^{-r \cosh u} \sinh (u) \d u = \frac{4\sqrt{2}}{\sigma^3\beta^3\sqrt{\pi}} K_0(r).
\eeqnn
Substituting this exact limit back into \eqref{26040901} yields the time-independent kernel $G(y,v)$:
\beqnn
\lim_{t \to \infty} t^{3/2} \e^{\frac{{\bf m}^2}{2\sigma^2} t} G_t(y, v) \ar=\ar 2^{\frac{\eta}{2}-1} y^{\frac{\eta}{2}-1} v^{\frac{\eta}{2}-1} \e^{-\frac{1}{2y} - v} \left[ \frac{4\sqrt{2}}{\sigma^3\beta^3\sqrt{\pi}} K_0\left(\sqrt{\frac{2v}{y}}\right) \right] \cr\cr
\ar=\ar \frac{4\sqrt{2}}{\sigma^3\beta^3\sqrt{\pi}} G(y, v).
\eeqnn

Furthermore, integrating the explicit upper bound of the integrand over $u \in (0, \infty)$ provides a global, time-independent bound for the entire temporal component:
\beqnn
\left| t^{3/2} \theta(r, \tau_t) \right| \ar\le\ar \frac{4\sqrt{2} r}{\sigma^3 \beta^3 \sqrt{\pi}} \int_0^{\infty} C u \e^{-r \cosh u}\sinh (u) \d u \cr\cr
\ar=\ar \frac{4\sqrt{2} C r}{\sigma^3 \beta^3 \sqrt{\pi}} \left( \frac{K_0(r)}{r} \right) \cr\cr
\ar=\ar \frac{4\sqrt{2} C}{\sigma^3 \beta^3 \sqrt{\pi}} K_0(r) := C' K_0(r).
\eeqnn
By replacing the temporal component with this uniform bound, we directly obtain the global upper bound for the scaled kernel:
\beqnn
t^{3/2} \e^{\frac{{\bf m}^2}{2 \sigma^2} t} G_t(y, v) \ar=\ar 2^{\frac{\eta}{2}-1} y^{\frac{\eta}{2}-1} v^{\frac{\eta}{2}-1} \e^{-\frac{1}{2y} - v} \left[ t^{3/2} \theta\left(\sqrt{\frac{2v}{y}}, \tau_t\right) \right] \cr\cr
\ar\le\ar 2^{\frac{\eta}{2}-1} y^{\frac{\eta}{2}-1} v^{\frac{\eta}{2}-1} \e^{-\frac{1}{2y} - v} \left[ C' K_0\left(\sqrt{\frac{2v}{y}}\right) \right] \cr\cr
\ar=\ar C' G(y,v).
\eeqnn
\qed
\end{proof}
\begin{lemma}\label{lem_limit_F}
In the weakly subcritical regime $({\bf m}\in(-\sigma^2,0)).$ Using the time-dependent function $G_t(y, v)$ defined in \eqref{Gt_def_time}, let $\nu_{\Phi, t}$ be a time-dependent measure on $\mbb{R}$ with its density defined precisely by $G_t$:
\beqlb\label{nu_phi_t_def}
\frac{\nu_{\Phi, t}(\d x)}{\d x} := \frac{\beta^2 \lambda^\beta}{2} \int_0^{\gamma^{-1/\beta}} \e^{-\beta(1-\eta)x} w^{-1} G_t\left(\e^{-\beta x}w^\beta, \frac{\lambda^\beta}{2}(w^{-\beta}-\gamma)\right) \d w.
\eeqlb
Then $F(t) = \mbb{E}^z[1-\e^{-\lambda Z_t}]$ can be exactly represented as an integral with respect to this measure:
\beqnn
F(t) = \int_{-\infty}^\infty \bigl(1 - \e^{-z \e^x}\bigr) \e^{-\beta\eta x}\,\nu_{\Phi, t}(\d x).
\eeqnn
Furthermore, as $t \to \infty$, the scaled density of the measure converges pointwise:
\beqnn
\lim _{t \rightarrow \infty} t^{3/2} \e^{\frac{{\bf m}^2}{2 \sigma^2} t} \frac{\nu_{\Phi, t}(\d x)}{\d x} = \frac{4\sqrt 2}{\sigma^3\beta^3\sqrt{ \pi }} \frac{\nu_\Phi(\d x)}{\d x},
\eeqnn
and consequently,
\beqlb\label{F_limit_conv}
\lim _{t \rightarrow \infty} t^{3/2} \e^{\frac{{\bf m}^2}{2 \sigma^2} t} F(t) = \frac{4\sqrt 2}{\sigma^3\beta^3\sqrt{ \pi }} \Phi(z, \lambda).
\eeqlb
\end{lemma}

\begin{proof}
By Theorem \ref{firstlem} and the joint density function provided in Proposition \ref{jtpdf}, $F(t)$ is given by:
\beqnn
F(t)\ar =\ar \int_{-\infty}^{\infty}\int_0^\infty \left[1-\exp\left(-z\left(\gamma y+{\lambda}^{-\beta}\e^{2 u} \right)^{-1/\beta}\right)\right]\cr\cr\ar\ar\qquad \e^{\eta u- \frac{\eta^2\sigma^2\beta^2}{8}t} \frac{1}{y} \exp \left(-\frac{1+\e^{2 u}}{2 y}\right) \theta\left(\frac{\e^u}{y}, \frac{\sigma^2\beta^2}{4}t\right) \d y \d u.
\eeqnn
We evaluate the integral by applying the change of variables $(y, u) \mapsto (x, w)$ directly:
\beqnn
y = \e^{-\beta x} w^\beta, \qquad u = \frac{\beta}{2}\ln\lambda - \frac{\beta}{2}x + \frac{1}{2}\ln(1-\gamma w^\beta),
\eeqnn
where the new domain is $x \in \mathbb{R}$ and $w \in (0, \gamma^{-1/\beta})$. The absolute value of the Jacobian determinant is computed as
\beqnn
\left| \frac{\partial(y,u)}{\partial(x,w)} \right| \ar=\ar \left| \left(-\beta \e^{-\beta x} w^\beta\right) \left(-\frac{\gamma \beta w^{\beta-1}}{2(1-\gamma w^\beta)}\right) - \left(\beta \e^{-\beta x} w^{\beta-1}\right)\left(-\frac{\beta}{2}\right) \right| \cr\cr
\ar=\ar \frac{\beta^2 \e^{-\beta x} w^{\beta-1}}{2(1-\gamma w^\beta)}.
\eeqnn
Thus, the transformation of the differential elements is given by:
\beqnn
\d y \d u = \left| \frac{\partial(y,u)}{\partial(x,w)} \right| \d x \d w = \frac{\beta^2 \e^{-\beta x} w^{\beta-1}}{2(1-\gamma w^\beta)} \d x \d w.
\eeqnn
Recalling that $-\frac{\eta^2\sigma^2\beta^2}{8} = -\frac{{\bf m}^2}{2\sigma^2}$, let $v = \frac{\e^{2u}}{2y} = \frac{\lambda^\beta}{2}(w^{-\beta}-\gamma)$. We note that $\e^{2u} = 2yv$. We substitute this into the integrand, extract $G_t$, and directly cancel the complexities of the Jacobian:
\beqnn
\ar\ar \e^{\eta u} \frac{1}{y} \exp \left(-\frac{1+\e^{2 u}}{2 y}\right) \e^{-\frac{{\bf m}^2}{2\sigma^2}t} \theta\left(\frac{\e^u}{y}, \frac{\sigma^2\beta^2}{4}t\right) \d y \d u \cr\cr
\ar=\ar (2yv)^{\frac{\eta}{2}} \frac{1}{y} \e^{-\frac{1}{2y} - v} \e^{-\frac{{\bf m}^2}{2\sigma^2}t} \theta\left(\sqrt{\frac{2v}{y}}, \frac{\sigma^2\beta^2}{4}t\right) \d y \d u \cr\cr
\ar=\ar \left[ 2^{\frac{\eta}{2}-1} y^{\frac{\eta}{2}-1} v^{\frac{\eta}{2}-1} \e^{-\frac{1}{2y} - v} \e^{-\frac{{\bf m}^2}{2\sigma^2}t} \theta\left(\sqrt{\frac{2v}{y}}, \frac{\sigma^2\beta^2}{4}t\right) \right] (2v) \, \d y \d u \cr\cr
\ar=\ar G_t(y, v) \cdot (2v) \, \d y \d u \cr\cr
\ar=\ar G_t(y, v) \left[ \lambda^\beta \left(w^{-\beta}-\gamma\right) \right] \left[ \frac{\beta^2 \e^{-\beta x} w^{\beta-1}}{2(1-\gamma w^\beta)} \d x \d w \right] \cr\cr
\ar=\ar G_t(y, v) \left[ \lambda^\beta \frac{1-\gamma w^\beta}{w^\beta} \right] \left[ \frac{\beta^2 \e^{-\beta x} w^{\beta-1}}{2(1-\gamma w^\beta)} \d x \d w \right] \cr\cr
\ar=\ar G_t(y, v) \left( \frac{\beta^2 \lambda^\beta}{2} \e^{-\beta x} w^{-1} \right) \d x \d w.
\eeqnn
Applying this algebraic identity to the original integral, noting from the substitution that $\gamma y + \lambda^{-\beta}\e^{2u} = \e^{-\beta x}$, and factoring $\e^{-\beta x} = \e^{-\beta(1-\eta)x} \e^{-\beta\eta x}$, the transformation completes immediately:
{\small\beqnn
F(t) \ar=\ar \int_{-\infty}^\infty \int_0^{\gamma^{-1/\beta}} \left[1-\exp\left(-z \e^x\right)\right] G_t\left(\e^{-\beta x}w^\beta, \frac{\lambda^\beta}{2}(w^{-\beta}-\gamma)\right) \left(\frac{\beta^2 \lambda^\beta}{2} \e^{-\beta x} w^{-1}\right) \d w \d x \cr\cr
\ar=\ar \int_{-\infty}^\infty \bigl(1 - \e^{-z \e^x}\bigr) \e^{-\beta\eta x} \left[ \frac{\beta^2 \lambda^\beta}{2} \int_0^{\gamma^{-1/\beta}} \e^{-\beta(1-\eta)x} w^{-1} G_t\left(\e^{-\beta x}w^\beta, \frac{\lambda^\beta}{2}(w^{-\beta}-\gamma)\right) \d w \right] \d x.
\eeqnn}
The term in the square brackets is precisely the measure density $\nu_{\Phi, t}(\d x)$ defined in \eqref{nu_phi_t_def}, resulting in the following:
\beqnn
F(t) = \int_{-\infty}^\infty \bigl(1 - \e^{-z \e^x}\bigr) \e^{-\beta\eta x}\,\nu_{\Phi, t}(\d x).
\eeqnn

To analyze the asymptotic behavior of the integral $F(t)$, we first establish the convergence of the scaled measure density. Recall the global uniform bound $t^{3/2} \e^{\frac{{\bf m}^2}{2 \sigma^2} t} G_t(y, v) \le C' G(y,v)$ established in Lemma \ref{lem_Gt_asymp}. Substituting the pointwise limit of $G_t$ into \eqref{nu_phi_t_def} and applying the dominated convergence theorem, we explicitly establish the exact convergence of the density.
{\small\beqnn
\lim _{t \rightarrow \infty} t^{3/2} \e^{\frac{{\bf m}^2}{2 \sigma^2} t} \frac{\nu_{\Phi, t}(\d x)}{\d x}
\ar=\ar \lim _{t \rightarrow \infty} \frac{\beta^2 \lambda^\beta}{2} \int_0^{\gamma^{-1/\beta}} \e^{-\beta(1-\eta)x} w^{-1} \left[ t^{3/2} \e^{\frac{{\bf m}^2}{2 \sigma^2} t} G_t\left(\e^{-\beta x}w^\beta, \frac{\lambda^\beta}{2}(w^{-\beta}-\gamma)\right) \right] \d w \cr\cr
\ar=\ar \frac{4\sqrt{2}}{\sigma^3\beta^3\sqrt{\pi}} \int_0^{\gamma^{-1/\beta}} \frac{\beta^2 \lambda^\beta}{2} \e^{-\beta(1-\eta)x} w^{-1} G\left(\e^{-\beta x}w^\beta, \frac{\lambda^\beta}{2}(w^{-\beta}-\gamma)\right) \d w \cr\cr
\ar=\ar \frac{4\sqrt 2}{\sigma^3\beta^3\sqrt{ \pi }} \frac{\nu_\Phi(\d x)}{\d x}.
\eeqnn}

With the pointwise limit of the scaled density established, we now consider the scaled $F(t)$:
\beqnn
t^{3/2} \e^{\frac{{\bf m}^2}{2 \sigma^2} t} F(t) = \int_{-\infty}^\infty \bigl(1 - \e^{-z \e^x}\bigr) \e^{-\beta\eta x} \left( t^{3/2} \e^{\frac{{\bf m}^2}{2 \sigma^2} t} \frac{\nu_{\Phi, t}(\d x)}{\d x} \right) \d x.
\eeqnn
Using the uniform bound $t^{3/2} \e^{\frac{{\bf m}^2}{2 \sigma^2} t} G_t(y, v) \le C' G(y,v)$ in \eqref{nu_phi_t_def} yields a uniform dominating bound for the scaled density itself:
\beqnn
t^{3/2} \e^{\frac{{\bf m}^2}{2 \sigma^2} t} \frac{\nu_{\Phi, t}(\d x)}{\d x} \le C' \frac{\nu_{\Phi}(\d x)}{\d x}.
\eeqnn
Applying this bound, the absolute value of the entire integrand is bounded by $$C' \bigl(1 - \e^{-z \e^x}\bigr) \e^{-\beta\eta x} \frac{\nu_{\Phi}(\d x)}{\d x}.$$ The integral of this dominant function is precisely $C' \Phi(z, \lambda) < \infty$. Therefore, by the dominated convergence theorem, we pass the limit inside the integral over $x$, giving \eqref{F_limit_conv} and completing the proof.
\qed
\end{proof}
\begin{lemma}\label{lem_survival_prob}
In the weakly subcritical regime $({\bf m}\in(-\sigma^2,0))$, the survival probability of the process satisfies the asymptotic relationship:
\beqnn
\lim _{t \rightarrow \infty}t^{3/2} \e^{\frac{{\bf m}^2}{2 \sigma^2} t}\mbb{P}^z\left(Z_t>0\right) = \frac{4\sqrt 2}{\sigma^3\beta^3\sqrt{ \pi }} \Psi(z).
\eeqnn
\end{lemma}

\begin{proof}
By definition, the survival probability can be obtained from the fractional expectation taking $\lambda \to \infty$: $\mbb{P}^z(Z_t>0) = \lim_{\lambda \to \infty} F(t)$. Using the exact integral representation of $F(t)$ established in Lemma \ref{lem_limit_F}, we have
\beqnn
\mbb{P}^z(Z_t>0) = \lim_{\lambda \to \infty} \int_{-\infty}^\infty \bigl(1 - \e^{-z \e^x}\bigr) \e^{-\beta\eta x} \nu_{\Phi, t}(\d x).
\eeqnn
To explicitly evaluate this limit, we recall the definition of time-dependent density $\nu_{\Phi, t}$ from \eqref{nu_phi_t_def} and apply the change in variable $v = \frac{\lambda^\beta}{2}(w^{-\beta}-\gamma)$. Noting that $w^\beta = (\gamma + 2v\lambda^{-\beta})^{-1}$ and that the combination of differential and reversing integration limits yields $\frac{\beta^2 \lambda^\beta}{2} w^{-1} \d w = \beta (\gamma + 2v\lambda^{-\beta})^{-1} \d v$, we can rewrite the density as:
\beqnn
\frac{\nu_{\Phi, t}(\d x)}{\d x} \ar=\ar \int_0^{\gamma^{-1/\beta}} \e^{-\beta(1-\eta)x} G_t\left(\e^{-\beta x}w^\beta, \frac{\lambda^\beta}{2}(w^{-\beta}-\gamma)\right) \left( \frac{\beta^2 \lambda^\beta}{2} w^{-1} \d w \right) \cr\cr
\ar=\ar \int_0^\infty \e^{-\beta(1-\eta)x} G_t\left(\e^{-\beta x}(\gamma + 2v\lambda^{-\beta})^{-1}, v\right) \left[ \beta (\gamma + 2v\lambda^{-\beta})^{-1} \right] \d v \cr\cr
\ar=\ar \int_0^\infty \beta \e^{-\beta(1-\eta)x} (\gamma + 2v\lambda^{-\beta})^{-1} G_t\left(\e^{-\beta x}(\gamma + 2v\lambda^{-\beta})^{-1}, v\right) \d v.
\eeqnn

Therefore, by the dominated convergence theorem, we pass the limit $\lambda \to \infty$ inside the integral, yielding the following:
\beqnn
\lim_{\lambda \to \infty} \frac{\nu_{\Phi, t}(\d x)}{\d x} = \int_0^\infty \beta \gamma^{-1} \e^{-\beta(1-\eta)x} G_t\left(\gamma^{-1} \e^{-\beta x}, v\right) \d v := \frac{\nu_{\Psi, t}(\d x)}{\d x}.
\eeqnn
Substituting this measure back, the survival probability becomes an exact expectation free of $\lambda$:
\beqnn
\mbb{P}^z(Z_t>0) = \int_{-\infty}^\infty \bigl(1 - \e^{-z \e^x}\bigr) \e^{-\beta\eta x} \nu_{\Psi, t}(\d x).
\eeqnn
Finally, we apply the temporal scaling $t^{3/2} \e^{\frac{{\bf m}^2}{2 \sigma^2} t}$. By Lemma \ref{lem_Gt_asymp},  $t^{3/2} \e^{\frac{{\bf m}^2}{2\sigma^2}t} G_t(y,v)$ is bounded by a time-independent integrable function and converges pointwise to $\frac{4\sqrt{2}}{\sigma^3\beta^3\sqrt{\pi}} G(y,v)$. Applying this property to the density of $\nu_{\Psi, t}(\d x)$ and applying the dominated convergence theorem, the scaled measure converges weakly to the reference measure $\nu_\Psi$:
\beqnn
\lim _{t \rightarrow \infty} t^{3/2} \e^{\frac{{\bf m}^2}{2 \sigma^2} t} \nu_{\Psi, t}(\d x) = \frac{4\sqrt 2}{\sigma^3\beta^3\sqrt{ \pi }} \nu_\Psi(\d x).
\eeqnn
Substituting this weak convergence directly into the integral representation of the scaled survival probability, we arrive at the final result:
\beqnn
\lim _{t \rightarrow \infty}t^{3/2} \e^{\frac{{\bf m}^2}{2 \sigma^2} t}\mbb{P}^z\left(Z_t>0\right) \ar=\ar \int_{-\infty}^\infty \bigl(1 - \e^{-z \e^x}\bigr)\e^{-\beta\eta x} \left( \lim _{t \rightarrow \infty}t^{3/2} \e^{\frac{{\bf m}^2}{2 \sigma^2} t} \nu_{\Psi, t}(\d x) \right) \cr\cr
\ar=\ar \frac{4\sqrt 2}{\sigma^3\beta^3\sqrt{ \pi }}\int_{-\infty}^\infty \bigl(1 - \e^{-z \e^x}\bigr) \e^{-\beta\eta x}\,\nu_\Psi(\d x) \cr\cr
\ar=\ar \frac{4\sqrt 2}{\sigma^3\beta^3\sqrt{ \pi }}\Psi(z).
\eeqnn

\qed
\end{proof}
\begin{proposition}\label{wkfenzi}
In the weakly subcritical regime $({\bf m}\in(-\sigma^2,0))$,
$$\lim_{t\rightarrow\infty}\mathcal L_t(z,\lambda)= 1 - \left(\frac{\gamma \lambda^{\beta}}{2}\right)^{\frac{\eta}{2}} U\left(\frac{\eta}{2}, 1, \frac{\gamma \lambda^{\beta}}{2}\right).$$
\end{proposition}

\begin{proof}
By \eqref{yglmt},
\beqnn
\lim_{t\rightarrow\infty}\mathcal L_t(z,\lambda)\ar =\ar \lim_{t \to \infty} \mbb{E}^{z}\left[\exp(-\lambda Z_t) \mid Z_t>0\right]\cr\cr\ar =\ar 1 - \frac{\lim_{t \to \infty} t^{3/2} \e^{\frac{{\bf m}^2}{2 \sigma^2} t} \left(1-\mbb{E}^{z}[\exp(-\lambda Z_t)]\right)}{\lim_{t \to \infty} t^{3/2} \e^{\frac{{\bf m}^2}{2 \sigma^2} t} \mbb{P}^z\left(Z_t>0\right)}.
\eeqnn
Substituting the explicit limits established in \eqref{F_limit_conv} of Lemma \ref{lem_limit_F} and Lemma \ref{lem_survival_prob}, the identical temporal scaling factors and the constants cancel exactly:
\beqnn
\lim_{t\rightarrow\infty}\mathcal L_t(z,\lambda) = 1 - \frac{\frac{4\sqrt 2}{\sigma^3\beta^3\sqrt{ \pi }} \Phi(z,\lambda)}{\frac{4\sqrt 2}{\sigma^3\beta^3\sqrt{ \pi }} \Psi(z)} = 1 - \frac{\Phi(z,\lambda)}{\Psi(z)}.
\eeqnn
Applying Corollary \ref{cor0409}, we obtain
\beqnn
\lim_{t\rightarrow\infty}\mathcal L_t(z,\lambda)\ar=\ar 1 - \left(\frac{\gamma \lambda^{\beta}}{2}\right)^{\frac{\eta}{2}} U\left(\frac{\eta}{2}, 1, \frac{\gamma \lambda^{\beta}}{2}\right),
\eeqnn
which completes the proof.
\qed
\end{proof}

\subsection{The intermediately subcritical and strongly subcritical regime}
The following two lemmas are crucial in both the intermediately and strongly subcritical regimes.
\begin{lemma}\label{fenzifenmu}
 Let $\left(X_t\right)_{t >0}$ and $\left(Y_t\right)_{t >0}$ be two families of non-negative, real-valued random variables. Assume that there exist functions from $(0, \infty)$ to $(0, \infty)$ denoted as $c_1(t)$ and $c_2(t)$, respectively, and constants $a_1, a_2 \in(0, \infty)$ such that
$$\lim _{t \rightarrow \infty} c_1(t) \mathbb E X_t=a_1,\quad \lim _{t \rightarrow \infty} c_2(t) \mathbb E Y_t=a_2,$$
$$
\limsup _{t \rightarrow \infty} c_1(t) \mathbb E X_t^2=0, \quad\limsup _{t \rightarrow \infty} c_2(t) \mathbb E Y_t^2=0 ,$$
and $\lim _{t \rightarrow \infty} \frac{c_2(t)}{c_1(t)} > 0$ exists. Then for any $z>0,$ we have
$$
\lim _{t \rightarrow \infty} \frac{\mathbb E\left[1-\e^{-z X_t}\right]}{\mathbb E\left[1-\e^{-z Y_t}\right]}=\lim_{t \rightarrow \infty} \frac{a_1 c_2(t)}{a_2 c_1(t)}.
$$
\end{lemma}
\begin{proof}
For fixed $z \in(0, \infty)$, we have
$$
z x-\frac{z^2 x^2}{2} \leqslant 1-\e^{-z x} \leqslant z x, \quad x>0.
$$
Thus, under the condition $\lim _{t \rightarrow \infty} c_1(t) \mathbb E X_t=a_1, \limsup _{t \rightarrow \infty} c_1(t)\mathbb E X_t^2=0$, we have
$$
\lim _{t \rightarrow \infty} c_1(t)\mathbb E\left[1-\exp \left(-z X_t\right)\right]=z a_1.
$$
Similarly, $$\lim _{t \rightarrow \infty} c_2(t)\mathbb E\left[1-\exp \left(-z Y_t\right)\right]=z a_2.$$
Therefore, if $\lim _{t \rightarrow \infty} \frac{c_2(t)}{c_1(t)}>0$ exists,
$$
\lim _{t \rightarrow \infty} \frac{\mathbb E\left[1-\e^{-z X_t}\right]}{\mathbb E\left[1-\e^{-z Y_t}\right]}=\frac{a_1}{a_2} \lim _{t \rightarrow \infty} \frac{c_2(t)}{c_1(t)}.
$$
\qed
\end{proof}

Recall that using Theorem \ref{firstlem}, the Laplace transform of the Yaglom limit is reduced to the problem of the ratio of the Laplace transforms of the two processes $(\xi_t(\lambda))_{t\ge 0}$ and $(\zeta_t)_{t\ge 0}$:
 \[\lim_{t\rightarrow\infty}\mathcal L_t(z,\lambda)=1-\lim _{t \rightarrow \infty}\frac{\mathbb E\left[1-\e^{-z \xi_t(\lambda)} \right]}{\mathbb E\left[1-\e^{-z \zeta_t} \right]},\quad z>0,\,\,\lambda>0.\]
 Now using Lemma \ref{fenzifenmu}, we can further simplify the problem to a problem that involves the first and second moments of the two processes $(\xi_t(\lambda))_{t\ge 0}$ and $(\zeta_t)_{t\ge 0}.$ As seen from the definition \eqref{0412xi}, $\xi_t(\lambda)$ is constructed from both $A^{(\eta)}_{t}$ and $B^{(e)}_{t}$. In principle, computing its first moment would require the joint distribution of the two processes. However, the following lemma allows us to rely exclusively on results concerning the limiting behavior of $A^{(\frac2\beta-\eta)}_{t}$.
\begin{lemma}\label{betatrans}For every $\lambda \in(0, \infty)$, $t \in(0, \infty)$ and $\eta \in \mathbb{R},$ we have
\begin{equation}\label{241127}
\mathbb E\xi_t(\lambda)=\e^{(\frac{1}{\beta}-\eta)\frac{\beta\sigma^2}{2}t} \mathbb{E}\left[\left(\gamma A_{\frac{\sigma^2\beta^2}{4}t}^{(\frac{2}{\beta}-\eta)}+{\lambda}^{-\beta}\right)^{-\frac{1}{\beta}}\right],
\end{equation}
and
\begin{equation}\label{2411270}
\mathbb E\zeta_t=\e^{(\frac{1}{\beta}-\eta)\frac{\beta\sigma^2}{2}t} \mathbb{E}\left[\left(\gamma A_{\frac{\sigma^2\beta^2}{4}t}^{(\frac{2}{\beta}-\eta)}\right)^{-\frac{1}{\beta}}\right],
\end{equation}
where $(\xi_t(\lambda))_{t\ge 0}$ and $(\zeta_t)_{t\ge 0}$ are defined \eqref{0412xi}.
\end{lemma}
\begin{proof}  Fix $T\ge 0.$ Since $\left(\exp\left\{\frac{2}{\beta}B^{(e)}_s-\frac{2}{\beta^2}s\right\}\right)_{0\le s\le T}$ is a martingale with respect to $(\mathcal F_s)_{0\le s\le T},$ there exists a probability $\mathbb Q$ satisfying $\frac{d\mathbb Q}{d\mathbb P}|_{\mathcal F_s}=\exp\left\{\frac{2}{\beta}B^{(e)}_s-\frac{2}{\beta^2}s\right\},$ for $s\in[0,T].$ Denote by $\mathbb E_\mathbb Q$ the expectation with respect to $\mathbb Q.$ Let $\mathbf C([0,T],\mathbb R)$ denote the space of all continuous functions from $[0,T]$ to $\mathbb R,$ endowed with the Borel \(\sigma\)-field generated by the topology of uniform convergence. We have for any Borel function $h:\mathbf C([0,T],\mathbb R)\rightarrow[0,\infty),$
\begin{equation*}\label{l251}\mathbb E\left[h\left((B^{(e)}_s)_{0\le s\le T}\right)\exp\left\{\frac{2}{\beta}B^{(e)}_T-\frac{2}{\beta^2}T\right\}\right]=\mathbb E_\mathbb Q\left[h\left((B^{(e)}_s)_{0\le s\le T}\right)\right], \end{equation*}
Using the Cameron-Martin-Girsanov theorem, see  \cite[Theorem 3.5.1]{rw00}, we know that $(B^{(e)}_s-\frac{2}{\beta}s)_{ 0\le s\le T}$ is a standard Brownian motion under $\mathbb Q$. Therefore,
\begin{equation*}\label{l252}\mathbb E_\mathbb Q\left[h\left((B^{(e)}_s-\frac{2}{\beta}s)_{0\le s\le T}\right)\right]=\mathbb E\left[h\left((B^{(e)}_s)_{0\le s\le T}\right)\right].\end{equation*}
Thus,
$$\mathbb E\left[h\left(\left(B^{(e)}_s-\frac{2}{\beta}s\right)_{0\le s\le T}\right)\exp\left\{\frac{2}{\beta}B^{(e)}_T-\frac{2}{\beta^2}T\right\}\right]=\mathbb E\left[h\left((B^{(e)}_s)_{0\le s\le T}\right)\right].$$
Taking $T=\frac{\sigma^2\beta^2}{4}t,$ we have
{\small
\begin{equation}\label{241120}
\begin{aligned}
\mathbb{E}\,\xi_t(\lambda)
&= \mathbb{E}\Bigg[
\exp\left\{\frac{2}{\beta}B^{(e)}_{\frac{\beta^2\sigma^2}{4}t}
-\frac{\sigma^2}{2}t\right\}  \times
\left(
\gamma\int_0^{\frac{\beta^2\sigma^2}{4}t}
\exp\left(2\left(\eta-\frac{2}{\beta}\right)u+2B^{(e)}_u\right)\,\d u \right. \\
&\qquad\qquad\qquad\qquad\qquad\qquad   \left.
+\lambda^{-\beta}\exp\left\{2\left(\frac{\sigma^2\beta^2}{4}\left(\eta-\frac{2}{\beta}\right)t
+B^{(e)}_{\frac{\sigma^2\beta^2}{4}t}\right)\right\}
\right)^{-1/\beta}
\Bigg].
\end{aligned}
\end{equation}
}
By the symmetry and stationary increments property of Brownian motion, the right hand side of \eqref{241120} can be simplified  to
$$ \e^{-\frac{\sigma^2}{2}t}\mathbb{E}\left[\left({\gamma  \int_0^{\frac{\beta^2\sigma^2}{4}t} \exp \left(2 (\eta -\frac{2}{\beta})u+2 B^{(e)}_{\frac{\beta^2\sigma^2}{4}t-u}\right) \d u+{\lambda}^{-\beta} \exp \left\{(\eta-\frac{2}{\beta}) \frac{\sigma^2\beta^2}{2}t\right\}}\right)^{-\frac{1}{\beta}}\right].$$
By substitution $s:=\frac{\sigma^2\beta^2}{4}t-u,$ it can be rewritten as
{\small
\beqnn
\ar\ar \e^{-\frac{\sigma^2}{2}t}\mathbb{E}\left[\bigg(\gamma  \int_0^{\frac{\beta^2\sigma^2}{4}t} \exp\bigg\{2\big[(\eta-\frac{2}{\beta})(\frac{\sigma^2\beta^2}{4}t-s)+B^{(e)}_s\big]\bigg\}\d s+{\lambda}^{-\beta} \exp \left\{(\eta-\frac{2}{\beta}) \frac{\sigma^2\beta^2}{2}t\right\}\bigg)^{-\frac{1}{\beta}}\right]\cr\cr
\ar=\ar \e^{-\frac{\sigma^2}{2}t} \exp \left\{(-\eta+\frac{2}{\beta}) \frac{\sigma^2\beta}{2}t\right\}\mathbb{E}\left[\bigg(\gamma  \int_0^{\frac{\beta^2\sigma^2}{4}t} \exp\bigg\{2\big[(\eta-\frac{2}{\beta})(-s)+B^{(e)}_s\big]\bigg\}\d s+{\lambda}^{-\beta}\bigg)^{-\frac{1}{\beta}}\right]\cr\cr
\ar=\ar \e^{(\frac{1}{\beta}-\eta)\frac{\beta\sigma^2}{2}t} \mathbb{E}\left[\left(\gamma A_{\frac{\sigma^2\beta^2}{4}t}^{(\frac{2}{\beta}-\eta)}+{\lambda}^{-\beta}\right)^{-\frac{1}{\beta}}\right].
\eeqnn
}
Thus, we have proved \eqref{241127}. Letting $\lambda\to\infty$ in \eqref{241127}, we obtain \eqref{2411270}.
\qed
\end{proof}

\begin{remark}
Through this lemma, we can see that in the weakly subcritical regime, $A^{(\frac2\beta-\eta)}_{t}$ has a positive drift, which prevents
$\xi_t(\lambda)$ and $\zeta_t$
from satisfying the second-moment condition in Lemma \ref{fenzifenmu}. Hence, this proof strategy is applicable only to the intermediately subcritical and strongly subcritical regimes.
\end{remark}
In the intermediately subcritical regime, $\eta=\frac 2\beta,$ we need the following result concerning the limiting behavior of $A_t^{(0)}.$

\begin{lemma}[B\"{o}inghoff and Hutzenthaler \cite{BH12}]\label{bh12} Let $g:[0, \infty) \rightarrow[0, \infty)$ be a Borel measurable function. Assume for some constants $c, p \in(0, \infty)$ that $g(x) \leq c x^p$ for every $x \geq 0$. Then we get that
$$
\lim _{t \rightarrow \infty} \sqrt{t} \cdot \mathbb{E}\left[g\left(\frac{1}{2 A_t^{(0)}}\right)\right]=\int_0^{\infty} g(a) \cdot \frac{1}{\sqrt{2 \pi}} \frac{\e^{-a}}{a} \d a<\infty.
$$
\end{lemma}

\begin{proposition}\label{itmfenzi} In the intermediately subcritical regime $({\bf m}=-\sigma^2)$, for any $z,\lambda>0,$ $$\lim_{t\rightarrow\infty}\mathcal L_t(z,\lambda)=1-\left(\frac{\gamma }{2}\right)^{\frac{1}{\beta}}U\left(\frac{1}{\beta},1,\frac{\gamma \lambda^{\beta}}{2}\right)\lambda.$$
\end{proposition}
\begin{proof}
Recall that $(\xi_t(\lambda))_{t\ge 0}$ and $(\zeta_t)_{t\ge 0}$ are defined by \eqref{0412xi}. We prove this result by applying Lemma \ref{fenzifenmu} with
  $$X_t=\xi_t(\lambda),\quad Y_t= \zeta_t,\quad c_1(t)=c_2(t)=\sqrt{t} \e^{\frac{\sigma^2}{2} t}.$$
By \eqref{yglmt} and Theorem \ref{firstlem},
\beqnn\lim_{t\rightarrow\infty}\mathcal L_t(z,\lambda)=1-\lim _{t \rightarrow \infty}\frac{\mathbb E^z\left[1-\e^{-\lambda Z_t} \right]}{\mathbb P^z(Z_t>0)}=1-\lim _{t \rightarrow \infty} \frac{\mathbb E\left(1-\e^{-z \xi_t(\lambda)}\right)}{\mathbb E\left(1-\e^{-z \zeta_t}\right)}.\eeqnn
Applying \eqref{241127},
$$\lim _{t \rightarrow \infty} c_1(t) \mathbb E \xi_t(\lambda)=\lim _{t \rightarrow \infty}\sqrt{t}\e^{\frac{\sigma^2}{2} t}\e^{(\frac{1}{\beta}-\eta)\frac{\beta\sigma^2}{2}t} \mathbb{E}\left[\left(\gamma A_{\frac{\sigma^2\beta^2}{4}t}^{(\frac{2}{\beta}-\eta)}+{\lambda}^{-\beta} \right)^{-\frac{1}{\beta}}\right].$$
In the intermediately subcritical regime,  $\eta=-\frac{2 {\bf m}}{\beta\sigma^2}=\frac{2}{\beta}.$ Hence
$$\lim _{t \rightarrow \infty} c_1(t) \mathbb E \xi_t(\lambda)=\lim _{t \rightarrow \infty}  \sqrt{t} \mathbb{E}\left[\left(\gamma A_{\frac{\sigma^2\beta^2}{4}t}^{(0)}+{\lambda}^{-\beta} \right)^{-\frac{1}{\beta}}\right].$$
Using Lemma \ref{bh12} with $g(x)=\left(\frac{\gamma }{2x}+{\lambda}^{-\beta}\right)^{-1/\beta},$ we get
$$\lim _{t \rightarrow \infty}\left(\frac{\sigma^2\beta^2}{4}t\right)^{1/2}\mathbb{E}\left[\left(\gamma A_{\frac{\sigma^2\beta^2}{4}t}^{(0)}+{\lambda}^{-\beta} \right)^{-\frac{1}{\beta}}\right]=\int_0^{\infty}  \left(\frac{\gamma }{2a}+{\lambda}^{-\beta}\right)^{-1/\beta}\frac{1}{\sqrt{2\pi}}\cdot  \frac{\e^{-a}}{a} \d a.$$
Thus,
\begin{equation}\label{interfenzi}
\lim _{t \rightarrow \infty} c_1(t) \mathbb E \xi_t(\lambda)=\lim _{t \rightarrow \infty}\sqrt{t}  \mathbb{E}\left[\left(\gamma A_{\frac{\sigma^2\beta^2}{4}t}^{(0)}+{\lambda}^{-\beta} \right)^{-\frac{1}{\beta}}\right]=\frac{\sqrt 2}{\beta\sigma\sqrt\pi}\int_0^{\infty}  \left(\frac{\gamma }{2a}+{\lambda}^{-\beta}\right)^{-\frac{1}{\beta}}\cdot  \frac{\e^{-a}}{a} \d a.
\end{equation}
Applying \eqref{2411270} and using Lemma \ref{bh12} with $g(x)=\left(\frac{\gamma }{2x}\right)^{-1/\beta},$ we obtain
\beqlb\label{interfenmu}
\lim _{t \rightarrow \infty} c_2(t) \mathbb E \zeta_t\ar=\ar\lim _{t \rightarrow \infty}\sqrt{t}  \mathbb{E}\left[\left(\gamma A_{\frac{\sigma^2\beta^2}{4}t}^{(0)} \right)^{-\frac{1}{\beta}}\right]\cr\cr
\ar=\ar\frac{\sqrt 2}{\beta\sigma\sqrt\pi}\int_0^{\infty}  \left(\frac{\gamma }{2a}\right)^{-\frac{1}{\beta}}\cdot  \frac{\e^{-a}}{a} \d a= \frac{\sqrt 2}{\beta\sigma\sqrt\pi}\left(\frac{2}{\gamma }\right)^{\frac{1}{\beta}}\Gamma\left(\frac{1}{\beta}\right).
\eeqlb

It is clear that
$$
\limsup _{t \rightarrow \infty} c_1(t) \mathbb E \xi_t(\lambda)^2\le\limsup _{t \rightarrow \infty} c_2(t) \mathbb E \zeta_t^2=\limsup _{t \rightarrow \infty}\sqrt{t}\e^{\frac{\sigma^2}{2} t} \mathbb{E}\left[\left(\gamma A_{\frac{\sigma^2\beta^2}{4}t}^{(\eta)}\right)^{-\frac{2}{\beta}}\right].
$$
 Applying \cite[Lemma 1]{pp2} with $p=\frac{1}{\beta}$ and $\eta=\frac{2}{\beta}$, we obtain that for any $t\ge 0,$ the following relation holds:
 $$\mathbb E\left[\left(\gamma A_{\frac{\sigma^2\beta^2}{4}t}^{(\eta)}\right)^{-\frac{2}{\beta}}\right]\le \e^{-\frac{\sigma^2}{2} t} \bigg[\mathbb E\left(\gamma  A_{\frac{\sigma^2\beta^2}{8}t}^{(0)}\right)^{-\frac{1}{\beta}}\bigg]^2.$$
From equation \eqref{interfenmu} we can see that $ \bigg[\mathbb E\left(\gamma  A_{\frac{\sigma^2\beta^2}{8}t}^{(0)}\right)^{-\frac{1}{\beta}}\bigg]^2$ decays as $t^{-1}$ as $t\rightarrow\infty.$ Thus
\beqlb\label{moment2}
\limsup _{t \rightarrow \infty} c_1(t) \mathbb E \xi_t(\lambda)^2\le\limsup _{t \rightarrow \infty} c_2(t) \mathbb E \zeta_t^2{ \le}\lim _{t \rightarrow \infty}\sqrt{t}  \bigg[\mathbb E\left(\gamma  A_{\frac{\sigma^2\beta^2}{8}t}^{(0)}\right)^{-\frac{1}{\beta}}\bigg]^2=0.
\eeqlb
Equations \eqref{interfenzi}, \eqref{interfenmu} and \eqref{moment2} guarantee that the condition of Lemma \ref{fenzifenmu} is satisfied.
 $$a_1=\frac{\sqrt 2}{\beta\sigma\sqrt\pi}\int_0^{\infty}  \left(\frac{\gamma }{2a}+{\lambda}^{-\beta}\right)^{-\frac{1}{\beta}}\cdot  \frac{\e^{-a}}{a} \d a, \qquad a_2= \frac{\sqrt 2}{\beta\sigma\sqrt\pi}\left(\frac{2}{\gamma }\right)^{\frac{1}{\beta}}\Gamma\left(\frac{1}{\beta}\right).$$
Therefore,
$$
\lim _{t \rightarrow \infty} \frac{\mathbb E\left[1-\e^{-z \xi_t(\lambda)}\right]}{\mathbb E\left[1-\e^{-z \zeta_t}\right]}=\frac{a_1}{a_2}
=\left(\frac{\gamma }{2}\right)^{\frac{1}{\beta}}\left(\Gamma(\frac{1}{\beta})\right)^{-1}\int_0^{\infty}  \left(\frac{\gamma }{2a}+{\lambda}^{-\beta}\right)^{-\frac{1}{\beta}}\cdot  \frac{\e^{-a}}{a} \d a.
$$
Thus, by \eqref{yglmt} and Theorem \ref{firstlem},
\beqnn
\lim_{t\rightarrow\infty}\mathcal L_t(z,\lambda)=1-\left(\frac{\gamma }{2}\right)^{\frac{1}{\beta}}\left(\Gamma\left(\frac{1}{\beta}\right)\right)^{-1}\int_0^{\infty}  \left(\frac{\gamma }{2a}+{\lambda}^{-\beta}\right)^{-\frac{1}{\beta}}\cdot  \frac{\e^{-a}}{a} \d a.
\eeqnn
Using a change of variable $u=\frac{2a}{\gamma \lambda^\beta},$ we get
\beqnn
\lim_{t\rightarrow\infty}\mathcal L_t(z,\lambda)
\ar=\ar
1-\lambda
\frac{\left(\frac{\gamma }{2}\right)^{1/\beta}}{\Gamma(1/\beta)}
\int_0^\infty
\e^{-\frac{\gamma \lambda^\beta}{2}u}
u^{\frac1\beta-1}(1+u)^{-1/\beta}\,\d u\cr\cr
\ar=\ar 1-\left(\frac{\gamma }{2}\right)^{\frac{1}{\beta}}U\left(\frac{1}{\beta},1,\frac{\gamma \lambda^{\beta}}{2}\right)\lambda.
\eeqnn
\qed

\end{proof}

For the strongly subcritical regime, $\eta>\frac 2\beta,$ we need the following result of Dufresne \cite{df90}, see Yor \cite[ (5.h)]{yor92}.
\begin{lemma}\label{gam}
Let $\left(B_t\right)_{t \geq 0}$ be a standard Brownian motion. For all $b>0$,
$$
\left(2\int_0^{\infty} \exp \left(2 B_s-2b s\right) \d s\right)^{-1}
$$
follows a Gamma distribution with the shape parameter $b$ and the scale parameter $1.$
\end{lemma}
\begin{proposition}\label{strgprop}
In the strongly subcritical regime $({\bf m}\in(-\infty,-\sigma^2))$, for any $z,\lambda>0$,
\[
\lim_{t\to\infty}\mathcal L_t(z,\lambda)
=
1-\left(\frac{\gamma}{2}\right)^{\eta-\frac{1}{\beta}}
U\left(\eta-\frac1\beta,\eta-\frac2\beta+1,\frac{\gamma \lambda^\beta}{2}\right)
\lambda^{\beta\eta-1}.
\]
\end{proposition}
 \begin{proof}
Recall that $(\xi_t(\lambda))_{t\ge 0}$ and $(\zeta_t)_{t\ge 0}$ are defined \eqref{0412xi}. We will show that the condition of Lemma \ref{fenzifenmu} is satisfied with
$$X_t=\xi_t(\lambda),\quad Y_t= \zeta_t,\quad c_1(t)=c_2(t)=\e^{-\left({\bf m}+\frac{\sigma^2}{2}\right) t},$$
  $$a_1=\frac{\Gamma(\eta-\frac{1}{\beta})}{\Gamma(\eta-\frac{2}{\beta})}
\left(\frac{\gamma}{2}\right)^{\eta-\frac{2}{\beta}}\lambda^{\beta\eta-1} U\left(\eta-\frac1\beta,\eta-\frac2\beta+1,\frac{\gamma\lambda^\beta}{2}\right),\quad a_2=\left(\frac{2}{\gamma}\right)^{\frac{1}{\beta}}\frac{\Gamma(\eta-\frac{1}{\beta})}{\Gamma(\eta-\frac{2}{\beta})}.$$
Since $\eta=-\frac{2{\bf m}}{\beta\sigma^2},$ $(\frac{1}{\beta}-\eta)\frac{\sigma^2\beta}{2}t=({\bf m}+\frac{\sigma^2}{2})t,$ applying \eqref{241127}, we have
$$\lim _{t \rightarrow \infty} c_1(t) \mathbb E \xi_t(\lambda)=\lim _{t \rightarrow \infty} \e^{-\left({\bf m}+\frac{\sigma^2}{2}\right) t}\e^{(\frac{1}{\beta}-\eta)\frac{\beta\sigma^2}{2}t} \mathbb{E}\left[\left(\gamma A_{\frac{\sigma^2\beta^2}{4}t}^{(\frac{2}{\beta}-\eta)}+{\lambda}^{-\beta}\right)^{-\frac{1}{\beta}}\right].
$$
In the strongly subcritical regime, $\frac{{\bf m}}{\sigma^2}<-1,$  $\eta>\frac{2}{\beta}.$ Hence the dominated convergence theorem and Lemma \ref{gam} yield that
{\small$$
\lim _{t \rightarrow \infty}\mathbb{E}\left[\left(\gamma A_{\frac{\sigma^2\beta^2}{4}t}^{(\frac{2}{\beta}-\eta)}+{\lambda}^{-\beta}\right)^{-\frac{1}{\beta}}\right]=\mathbb{E}\lim _{t \rightarrow \infty}\left[\left(\gamma A_{\frac{\sigma^2\beta^2}{4}t}^{(\frac{2}{\beta}-\eta)}+{\lambda}^{-\beta} \right)^{-\frac{1}{\beta}}\right]=\mathbb{E}\left[\left(\frac{\gamma}{2G_{\eta-\frac{2}{\beta}}} +{\lambda}^{-\beta}\right)^{-\frac{1}{\beta}}\right],
$$}
where
$G_{\eta-\frac{2}{\beta}}:=\lim\limits_{t\rightarrow\infty}\left(2A_t^{(\frac{2}{\beta}-\eta)}\right)^{-1}$ is Gamma distributed with shape parameter $\eta-\frac{2}{\beta}$ and scale parameter $1.$
Using the density function of the gamma distribution, we get
{\small\beqnn
\mathbb{E}\left[\left(\frac{\gamma}{2G_{\eta-\frac{2}{\beta}}} +{\lambda}^{-\beta}\right)^{-\frac{1}{\beta}}\right]\ar=\ar\frac1{\Gamma(\eta-\frac{2}{\beta})}
\int_0^\infty
\left(
\frac{\gamma}{2x}+\lambda^{-\beta}
\right)^{-1/\beta}
x^{\eta-\frac{2}{\beta}-1}\e^{-x}\,\d x\cr\cr
\ar=\ar \frac1{\Gamma(\eta-\frac{2}{\beta})}\left(\frac{2}{\gamma}\right)^{1/\beta}
\int_0^\infty
\left(
\frac{1}{x}+\frac{2}{\gamma\lambda^{\beta}}
\right)^{-1/\beta}
x^{\eta-\frac{2}{\beta}-1}\e^{-x}\,\d x\cr\cr
\ar=\ar \frac1{\Gamma(\eta-\frac{2}{\beta})}\left(\frac{2}{\gamma}\right)^{1/\beta}
\int_0^\infty
\left(
1+\frac{2x}{\gamma\lambda^{\beta}}
\right)^{-1/\beta}
x^{\eta-\frac{1}{\beta}-1}\e^{-x}\,\d x\cr\cr
\ar=\ar \frac1{\Gamma(\eta-\frac{2}{\beta})}\left(\frac{2}{\gamma}\right)^{1/\beta}
\int_0^\infty
\left(
1+u
\right)^{-1/\beta}
\left(\frac{\gamma\lambda^\beta}{2}\right)^{\eta-\frac{1}{\beta}} u^{\eta-\frac{1}{\beta}-1}\exp\left(-\frac{\gamma\lambda^\beta}{2}u\right)\,\d u\cr\cr
\ar=\ar \frac{\Gamma(\eta-\frac{1}{\beta})}{\Gamma(\eta-\frac{2}{\beta})}
\left(\frac{\gamma}{2}\right)^{\eta-\frac{2}{\beta}}\lambda^{\beta\eta-1} U\left(\eta-\frac1\beta,\eta-\frac2\beta+1,\frac{\gamma\lambda^\beta}{2}\right).
\eeqnn}
Similarly,
\begin{equation*}
\lim _{t \rightarrow \infty} c_2(t) \mathbb E \zeta_t=\lim _{t \rightarrow \infty}\mathbb{E}\left[\left(\gamma A_{\frac{\sigma^2\beta^2}{4}t}^{(\frac{2}{\beta}-\eta)}\right)^{-\frac{1}{\beta}}\right]=\mathbb{E}\left[\left(\frac{\gamma}{2G_{\eta-\frac{2}{\beta}}}\right)^{-\frac{1}{\beta}}\right]=\left(\frac{2}{\gamma}\right)^{\frac{1}{\beta}}\frac{\Gamma(\eta-\frac{1}{\beta})}{\Gamma(\eta-\frac{2}{\beta})}.
\end{equation*}
On the other hand,
$$ \limsup _{t \rightarrow \infty} c_1(t)\mathbb E \xi^2_t\le\limsup _{t \rightarrow \infty} c_2(t)\mathbb E \zeta^2_t =\limsup _{t \rightarrow \infty}\e^{-\left({\bf m}+\frac{\sigma^2}{2}\right) t}\mathbb E\left[ \left(\gamma A_{\frac{\sigma^2\beta^2}{4}t}^{(\eta)}\right)^{-\frac{2}{\beta}}\right]=0.$$
To see this, using \cite[Lemma 1 (ii)]{pp}, we have
$$\mathbb E\left[\left(A_{\frac{\sigma^2\beta^2}{4}t}^{(\eta)}\right)^{-\frac{2}{\beta}}\right]\le \e^{\left({\bf m}+\frac{\sigma^2}{2}\right) t}\mathbb E\left[\left(A_{\frac{\sigma^2\beta^2}{8}t}^{(-\eta+\frac{2}{\beta})}\right)^{-\frac{1}{\beta}}\right]\mathbb E\left[ \left(A_{\frac{\sigma^2\beta^2}{8}t}^{(\eta-\frac{2}{\beta})}\right)^{-\frac{1}{\beta}}\right].$$
 Note that $A_t^{(\eta-\frac{2}{\beta})}$ goes to $\infty$ as $t$ increases. Moreover, for all \(t\ge \frac{8}{\sigma^2\beta^2}\),
\[
0\le
\left(A_{\frac{\sigma^2\beta^2}{8}t}^{(\eta-\frac{2}{\beta})}\right)^{-1/\beta}
\le
\left(A_1^{(\eta-\frac{2}{\beta})}\right)^{-1/\beta}\le
\exp\left(\frac{2}{\beta}\sup_{0\le s\le 1}|B_s^{(e)}|\right).
\]
Now
\[
\sup_{0\le s\le 1}|B_s^{(e)}|
\le
\sup_{0\le s\le 1}B_s^{(e)}+\sup_{0\le s\le 1}(-B_s^{(e)}).
\]
Hence, by the Cauchy--Schwarz inequality,
\[
\mathbb E\exp\left(\frac{2}{\beta}\sup_{0\le s\le 1}|B_s^{(e)}|\right)
\le
\left(
\mathbb E\exp\left(\frac{4}{\beta}\sup_{0\le s\le 1}B_s^{(e)}\right)
\right)^{1/2}
\left(
\mathbb E\exp\left(\frac{4}{\beta}\sup_{0\le s\le 1}(-B_s^{(e)})\right)
\right)^{1/2}.
\]
By the reflection principle for Brownian motion,
\[
\sup_{0\le s\le 1}B_s^{(e)}\stackrel{d}{=} |B_1^{(e)}|,
\qquad
\sup_{0\le s\le 1}(-B_s^{(e)})\stackrel{d}{=} |B_1^{(e)}|.
\]
Therefore,
\[
\mathbb E\exp\left(\frac{2}{\beta}\sup_{0\le s\le 1}|B_s^{(e)}|\right)
\le
\mathbb E\exp\left(\frac{4}{\beta}|B_1^{(e)}|\right)<\infty.
\]
Thus \(\left(A_1^{(\eta-\frac{2}{\beta})}\right)^{-1/\beta}\) is integrable, and the dominated convergence theorem yields
\[
\lim_{t\to\infty}
\mathbb E\left[
\left(A_{\frac{\sigma^2\beta^2}{8}t}^{(\eta-\frac{2}{\beta})}\right)^{-1/\beta}
\right]
=0.
\]

Therefore,
\beqnn\lim _{t \rightarrow \infty}\frac{\mathbb E^z\left[1-\e^{-\lambda Z_t} \right]}{\mathbb P^z(Z_t>0)}\ar=\ar\lim _{t \rightarrow \infty} \frac{\mathbb E\left[1-\e^{-z \xi_t(\lambda)}\right]}{\mathbb E\left[1-\e^{-z \zeta_t}\right]}=\frac{a_1}{a_2}
\cr\cr
\ar=\ar
\left(\frac{\gamma}{2}\right)^{\eta-\frac{1}{\beta}}\lambda^{\beta\eta-1} U\left(\eta-\frac1\beta,\eta-\frac2\beta+1,\frac{\gamma\lambda^\beta}{2}\right).
\eeqnn
The desired result then follows.
\qed
\end{proof}
\noindent{\it Proof of Theorem \ref{mainthm}.}
Putting the results of Propositions \ref{wkfenzi}, \ref{itmfenzi}, and \ref{strgprop} together yields \eqref{mainthm0327}. By \cite[Theorem 1.2]{li2020}, as $t\to\infty$, $\mbb{P}^z(Z_t\in\cdot\mid Z_t>0)$ converges weakly to a probability measure on $[0, \infty)$, which is exactly the Yaglom limit, and its associated Laplace transform is defined by \eqref{mainthm0327}.
\qed

\end{document}